\documentclass[12 pt]{amsart}
\usepackage{amssymb,amsfonts,amsthm}
\usepackage{graphicx}
\usepackage{color} 
\usepackage{bm}
\newtheorem{thm}{Theorem}[section]
\newtheorem{cor}[thm]{Corollary}
\newtheorem{lem}[thm]{Lemma}
\newtheorem{pro}[thm]{Proposition}

\theoremstyle{definition}

\theoremstyle{remark}
\newtheorem{rem}[thm]{Remark}
\numberwithin{equation}{section}

\begin{document}
\title[Simulate polynomial convergence numerically]{Numerical
  simulation of polynomial-speed convergence phenomenon}
\author{Yao Li}
\address{Yao Li: Department of Mathematics and Statistics, University
  of Massachusetts Amherst, Amherst, MA, 01002, USA}
\email{yaoli@math.umass.edu}

\author{Hui Xu}
\address{Hui Xu: Department of Mathematics, Amherst College, Amherst,
  MA, 01002, USA}
\email{huxu18@amherst.edu}
\thanks{Hui Xu was supported in part by REU program of University of
  Massachusetts Amherst. }

\begin{abstract}
  We provide a hybrid method that captures the polynomial speed of
  convergence and polynomial speed of mixing for Markov processes. The hybrid method that we introduce is based on the coupling technique and renewal theory. We propose to replace
  some estimates in classical results about the ergodicity of Markov
  processes by numerical simulations when the corresponding analytical proof is difficult. After that, all remaining conclusions can be derived from rigorous analysis. Then we
  apply our results to seek numerical justification for the ergodicity of two 1D microscopic heat conduction models. The
  mixing rate of these two models are expected to be polynomial but
  very difficult to prove. In both examples, our numerical results match the expected
  polynomial mixing rate well.

\end{abstract}
\maketitle

\section{Introduction}
The aim of this paper is two-fold. From the viewpoint of statistical mechanics, this paper
aims to justify the polynomial ergodicity of a class of 1D microscopic
heat conduction models. Purely rigorous analysis of
  polynomial ergodicity of these models using current analytical
  techniques fails to provide satisfactorily accurate results. This paper also
  aims to establish a comprehensive method that can be applied on a
  broader scale. That is, from the viewpoint of numerical analysis, we
  also want to propose a hybrid method that captures polynomial-speed
  convergence to steady-states for general Markov processes.

Heat conduction is ubiquitous in the universe and has been
well-studied at the macroscopic level. However, from a microscopic
point of view, the study of how energy is transported in materials is a very
challenging topic. In particular, the derivation of Fourier's law from microscopic
Hamiltonian dynamics is a century-old challenge to mathematicians and
physicists \cite{bonetto2000fourier, eckmann1999nonequilibrium,
  eckmann2006nonequilibrium, gaspard2008heat, rey2000asymptotic,
  bricmont2007towards}.  Due to the significant difficulty of analyzing Hamiltonian
models, many researchers seek stochastic approximations of Hamiltonian
dynamics in microscopic heat conduction models
\cite{eckmann2006nonequilibrium, li2014nonequilibrium,
  gaspard2008heat, kipnis1982heat, grigo2012mixing, derrida2002large}. In this
paper, we will work primarily on these stochastic heat conduction models.

When an 1D stochastic heat conduction model is connected to two
thermalized boundaries with different temperatures, one would expect
the existence of a naturally selected steady-state, called the
non-equilibrium steady-state (NESS). An analysis of topics
like existence and uniqueness of NESS, and speed of convergence to
NESS will open the door to further studies such as the thermal
conductivity, the existence of local thermodynamic equilibrium, the
Gallavotti-Cohen fluctuation theorem, and
eventually the Fourier's law. However, a rigorous analysis about the
ergodicity of the NESS is usually very challenging.

Due to complicated interactions within the chain, a
stochastic microscopic heat conduction model may have sub-exponential speed of
mixing and sub-exponential speed of convergence to the NESS. In this
paper we will present two 1D microscopic heat conduction models,
namely, the stochastic energy exchange model and the random halves
model, both of which originate from deterministic dynamical systems
\cite{eckmann2006nonequilibrium, bunimovich1992ergodic, li2014nonequilibrium}. One common
feature of these two models is that a low energy particle (or a low
energy site) requires a long time to have the next energy
exchange. Because of this, we expect the rate of convergence and the
rate of mixing to be $\sim t^{-2}$. We refer readers to Section 5 and
Section 6 for more engaged discussion about microscopic heat conduction models
and their connections to deterministic dynamical systems. Besides
these two models, other microscopic heat conduction models that have
sub-exponential mixing rate include the particle model in
\cite{yarmola2013sub, yarmola2014sub}, the rotor model in
\cite{cuneo2016non, cuneo2015non}, and
the anharmonic chain in \cite{hairer2009slow}. Other examples of sub-exponential rate of
convergence have also been observed in various models like MCMC algorithms
and random walks \cite{jarner2002polynomial, tuominen1994subgeometric,
mengersen1996rates}. 

Regardless of  the detailed setting of
models, a rigorous proof of slow mixing phenomena is known to be very
difficult, partially because all methods 
that use the spectral gap of the infinitesimal generator simply fail
to work. Without using spectrum analysis,
one needs to use probabilistic approaches. There is a very rich literature
about probabilistic methods of proving ergodicity of Markov
processes. We refer \cite{meyn2009markov, meyn1993stability,
  meyn1993stability2, hairer2011yet, hairer2010convergence} for results about exponential ergodicity and \cite{jarner2002polynomial, hairer2010convergence,
  douc2004practical, tuominen1994subgeometric} for results about
sub-exponential ergodicity. Almost all of these probabilistic methods require
a reference set in which independent trajectories can couple with a strictly
positive probability (called the minorization condition), and a
Lyapunov function that ``pushes'' trajectories to the reference set
(called the drift condition). The idea is that once entering the
uniform reference set, trajectories of the Markov process
can be coupled and becomes indistinguishable. This gives a quantitative
bound of the convergence speed in (weighted) total variation norm or other
weaker norms \cite{hairer2011asymptotic, hairer2008spectral}. However, there is no generic approach of
constructing such a Lyapunov function. It has to be done
in an ad hoc manner. If a Markov process lives on a high dimensional space,
such a construction is usually very difficult. Even if a rigorous proof is
possible for simpler models such as the stochastic energy exchange
model studied in this paper \cite{li2016polynomial2}, the bound of
convergence speed to the steady-state is usually not accurate,
partially because an explicit expression of the invariant probability
measure is usually not
possible.   

A numerical justification of slow mixing (and slow convergence)
phenomenon is challenging as well. Let $P^{t}$ be the transition kernel of a Markov process $X_{t}$. Let $f$ and $g$
be two observables on the state space of $X_{t}$. The decay of
correlation is denoted by
$$
  C_{\mu}(t) = | \int (P^{t} f)(x) g(x) \mu(\mathrm{d}x) - \int (P^{t}f(x) \mu(\mathrm{d}x)
  \int g(x) \mu(\mathrm{d}x) | \,,
$$
where $\mu$ is a probability measure. A direct simulation of
$C_{\mu}(t)$ requires Monte Carlo simulations of $\int (P^{t} f)(x)
g(x) \mu(\mathrm{d}x) $ and $\int (P^{t}f(x)
\mu(\mathrm{d}x)$. Therefore, it is easy to see that the estimator of
$C_{\mu}(t)$ has variance $O(1)$. Now assume $C_{\mu}(t)$ has a
polynomial tail $C_{\mu}(t) \sim t^{-\alpha}$. A simple calculation
shows that to make the relative error of $C_{\mu}(T)$ less than
$\epsilon$, the sample size should be at least $\epsilon^{-2}T^{2
  \alpha}$, which brings the total computational cost to
$\epsilon^{-2}T^{2 \alpha + 1}$. Our simulation shows that large $T$ is usually
necessary to effectively capture the tail of $C_{\mu}(t)$, which makes the
computational cost of direct Monte Carlo simulation unacceptable. For
instance, when $T = 10^{3}$, $\alpha = 2$, and $\epsilon = 0.1$, the
computational cost of direct simulation is $\sim 10^{17}$. In some
studies, the correlation decay is simply justified by computing the
convergence rate of a few selected observables
\cite{franzke2011noise}, which is unfortunately not a strong evidence
to support the argument about the rate of correlation decay. 

In fact, there are very limited literatures about numerical justifications
of convergence rate (or decay rate of correlation) of Markov
processes. Most known studies choose to numerically verify the drift
condition and the minorization condition \cite{jones2004sufficient,
  cowles1998simulation, roberts1998convergence}. Hence these results still rely on
a known Lyapunov function. As explained above, usually a Lyapunov function can only be
constructed in an ad hoc manner. This means that many difficulties in rigorous
proofs remain unsolved. In addition, numerically showing the drift
condition on the entire state space can be very expensive. There are
also known results about the convergence rate of the MCMC algorithm
\cite{athreya1996convergence}, which assumes that the invariant
probability measure is known. However, in most nonequilibrium systems,
an explicit expression of the invariant probability measure is not
possible.

In this paper, we present a hybrid method that combines the advantage
of both analytical and numerical methods to calculate the polynomial speed
of convergence for Markov processes. As an application of this method,
we numerically show that the two microscopic heat conduction models
have speed of convergence $\sim t^{-2}$ to their steady states. This
is consistent to both our heuristic analysis and numerical results for
corresponding deterministic models. 

Dated back to several decades ago, the early probabilistic approach of
proving convergence rate to the invariant probability measure is based on the coupling
method and discrete renewal theory \cite{lindvall1979coupling, nummelin1982geometric,
  nummelin1983rate}. We run two independent copies of the Markov process until they are coupled and become
indistinguishable. The rule of coupling is that when both processes
enter a certain set $\mathfrak{C}$ called the uniform reference 
set (or small set in some literatures), they have positive probability
to couple (the minorization condition). Then the coupling lemma tells us that the speed of
convergence in the total variation norm is mainly determined by the
tail of the first passage time to $\mathfrak{C}$, denoted by $\tau_{\mathfrak{C}}$. It is usually not difficult to construct such a uniform
reference set $\mathfrak{C}$. But an analytical estimation of
$\tau_{\mathfrak{C}}$ is usually difficult. The drift condition
investigated in numerous later literatures are used to estimate the first passage
time to $\mathfrak{C}$ \cite{hairer2010convergence, meyn2009markov, jarner2002polynomial}. 

The main strategy proposed in this paper is to numerically estimate
$\tau_{\mathfrak{C}}$ directly. This bypasses the difficulty of
constructing and working on Lyapunov functions. This is important
because in some complicated models, the mechanism of slow convergence
is not yet fully understood. From the numerical
analysis point of view, the first passage time of a Markov chain can be computed easily with very
high accuracy. After obtaining a numerical tail of
$\tau_{\mathfrak{C}}$, we use coupling technique and renewal theory to show that the
Markov process has a polynomial speed of convergence.

The organization of this paper is as follows. Section 2 serves as a probability
preliminary. Then we will discuss numerically or analytically verifiable
conditions that lead to a polynomial convergence rate in Section
3. Section 4 discusses conclusions that can be made from numerical and
analytical conditions. Some discussion about continuous-time Markov
process is also made in Section 4. Finally, we investigate two microscopic heat conduction models in
Section 5 and Section 6, respectively.

\section{Probability preliminary: convergence rates of Markov chains}
 
The purpose of this section is to review known sufficient conditions
towards the polynomial ergodicity of a Markov process. Throughout this
section, we let $\Psi_{n}$ be a discrete-time Markov chain on a
measurable space $(X, \mathcal{B})$. The transition kernel of
$\Psi_{n}$ is denoted by $\mathcal{P}(x, \cdot)$. For each $A \in
\mathcal{B}$, $P(\cdot, A)$ is a measurable function. For each $x \in
X$, $P(x, \cdot)$ is a probability measure.

For $A \in \mathcal{B}$, we let $\tau_{A}$ be the first passage time to $A$: 
$$
  \tau_{A} =  \inf \{ n > 0 \,| \, \Psi_{n} \in A \} \,.
$$
A set $A \in \mathcal{B}$ is said to be accessible if
$\mathbb{P}_{x}[\tau_{A} < \infty] = 1$ for every $x \in X$.

We say a Markov process is {\it irreducible} with respect to a measure
$\phi$ on $\mathcal{B}$ if every $A \in \mathcal{B}$ with $\phi(A) >
0$ is accessible. We refer readers to Chapter 4 of
\cite{meyn2009markov} for detailed definition and properties of irreducibility. If $\Psi_{n}$ is irreducible with
respect to a non-trivial measure $\phi$, then there exists a ``maximal
irreducible measure'' $\psi$ such that $\phi \ll \psi$, where $\psi$ is unique up to equivalence classes. We skip the formal introduction of the 
maximal irreducibility as the $\phi$-irreducibility is
sufficient for this paper.

{\bf (A) Construction of an atom.} It has long been known that the stochastic stability of $\Psi_{n}$,
such as recurrence, ergodicity, and decay of correlation, follows from
certain pseudo-atomic properties \cite{meyn2009markov, nummelin1983rate}. More precisely, we need a {\it
  uniform reference set} $\mathfrak{C} \in \mathcal{B}$ that satisfies
$$
  \sup_{x \in \mathfrak{C}} \mathcal{P}(x, \cdot) \geq \eta \theta(\cdot) \,,
$$
where $\theta$ is a probability measure on $(X, \mathcal{B})$ and
$\eta$ is a strictly positive real number. This is called the
``minorization condition'' in many literatures.

Assuming the existence of
such a uniform reference set $\mathfrak{C}$, the state space of
$\Psi_{n}$ can be splitted by letting $\tilde{X} = X \cup \mathfrak{C}_{1}$, where $\mathfrak{C}_{1}$ is an
identical copy of $\mathfrak{C}_{0} := \mathfrak{C}$. Then we can naturally extend
$(X, \mathcal{B})$ to $( \tilde{X}, \mathcal{\tilde{B}})$ and
``split'' a probability measure $\mu$ into a probability measure
$\mu^{*}$ on $(\tilde{X}, \tilde{\mathcal{B}})$: 

$$
  \left \{
\begin{array}{ll}
 & \mu^*|_X = (1 - \eta)\ \mu|_{\mathfrak{C}_0} + \mu|_{X \setminus \mathfrak{C}_0}\\
&  \mu^*|_{\mathfrak{C}_1} = \eta \ \mu|_{\mathfrak{C}_0}\ , \quad \mathfrak{C}_0 \cong \mathfrak{C}_1  \mbox{ via the natural
identification }.
\end{array}
\right .
$$

With the above construction, we can define the split chain
$\tilde{\Psi}$ on $(\tilde{X}, \tilde{\mathcal{B}})$ with a transition
kernel $\tilde{\mathcal{P}}(x, \cdot)$:

$$
  \left \{ 
\begin{array}{cl}
 \mathcal{\tilde P}(x, \cdot) = (\mathcal{P}(x, \cdot))^* &  x \in X \setminus \mathfrak{C}_0\\
 \mathcal{\tilde P}(\mathbf{x}, \cdot) = [(\mathcal{P}(x, \cdot))^{*}
 - \eta {\theta^{*}(\cdot)}]/(1 - \eta) 
 & x \in \mathfrak{C}_0 \\
\mathcal{\tilde P}( x, \cdot) = {\theta^{*}}(\cdot) & x \in \mathfrak{C}_1
\end{array}
\right .
$$

It is straightforward to check that $\tilde{\Psi}_{n}$ possesses an
atom $\mathfrak{C_1}$: for all $x \in \mathfrak{C}_{1}$, $\mathcal{P}(x, \cdot) =
\theta(\cdot)$. In addition, the natural projection $\pi: \tilde{X}
\rightarrow X$ projects $\tilde{\Psi}_{n}$ to $\Psi_{n}$. This split
construction is called the Nummelin splitting. We refer \cite{nummelin1978splitting}
for the detail.

The aperiodicity of $\Psi_{n}$ can follow from the properties of
$\mathfrak{C}$. If $\Psi_{n}$ is irreducible and admits a uniform
reference set such that $\theta( \mathfrak{C}) > 0$, $\Psi_{n}$ is
said to be {\it strongly aperiodic}. If $\Psi_{n}$ is strongly
aperiodic, $\Psi_{n}$ must be aperiodic such that no cyclic
decomposition is possible. We refer readers to Chapter 5 of \cite{meyn2009markov} for the
complete statement of aperiodicity of Markov processes.

The following theorem gives the existence of an invariant probability
measure of $\Psi_{n}$.

\begin{thm}[Theorem 10.0.1 from \cite{meyn2009markov}]
\label{existence}
Let $\Psi_{n}$ be an irreducible aperiodic Markov process on $(X,
\mathcal{B})$. If $\mathfrak{C} \in B$ is an accessible uniform
reference set such that 
$$
  \sup_{x \in \mathfrak{C}} \mathbb{E}_{x}[ \tau_{\mathfrak{C}} ] < \infty
$$
then there exists an invariant probability measure $\pi$. 
\end{thm}

{\bf (B) Coupling. } The speed of convergence of $\Psi_{n}$ then
follows from the following coupling argument. Without loss of
generality, assume that the uniform reference set $\mathfrak{C}$ is accessible, i.e., 
$$
\mathbb{P}_{x}[ \tau_{\mathfrak{C}} < \infty] := \mathbb{P}_{x}[\inf \{ n > 0
\,|\,\Psi_{n} \in \mathfrak{C} \} ]= 1 
$$
for any $x \in X$. Let $\mu$ and $\nu$ be two initial distributions. One way to bound $\|
\mu \mathcal{P}^{n} - \nu \mathcal{P}^{n} \|_{TV}$ as $n \rightarrow
\infty$ is to run two independent copies of $\tilde{\Psi}_{n}$
starting from $\mu^{*}$ and $\nu^{*}$, respectively, and perform a
coupling at their first simultaneous return to the atom $\mathfrak{C}_{1}$. Let
$T$ be the coupling time. It is well known that 
$$
  \|\mu \mathcal{P}^{n} - \nu \mathcal{P}^{n}\|_{TV} \leq \| \mu^{*}
  \mathcal{\tilde{P}}^{n} - \nu^{*} \mathcal{\tilde{P}}^{n } \|_{TV}
  \leq 2 \mathbb{P}[ T > n] \,.
$$
We refer \cite{lindvall2002lectures, hairer2010convergence} for
details of the coupling method.

Let $Y_{0}, Y_{1}, Y_{2}, \cdots$ and $Y'_{0}, Y'_{1}, Y'_{2}, \cdots$
be the passage times to $\mathfrak{C}_{1}$ for two independent
processes, respectively. It is obvious that $Y_{1}, Y_{2},
\cdots$ and $Y_{1}', Y_{2}', \cdots$ are {\it i.i.d} random variables
with a distribution $\tau_{\mathfrak{C}_{1}} |_{\mathfrak{C}_{1}}$. Therefore, the coupling
time $T$ is the first simultaneous renewal time for renewal processes 
$$
  S_{n} := \sum_{i = 0}^{n} Y_{i}, \quad \mathrm{and} \quad  S'_{n} :=
  \sum_{i = 0}^{n} Y'_{i} \,.
$$
More precisely, we have
$$
  T = \inf_{n \geq 0} \{ S_{i} = S'_{j} = n \mbox{ for some } i, j \} \,.
$$

If in addition, the return times to $\mathfrak{C}_{1}$ are aperiodic,
i.e., the greatest common divisor of $\{ n \geq 1 \,|\, \mathbb{P}[
Y_{i} = n] > 0 \}$ is $1$, then it follows from \cite{lindvall2002lectures} that the finiteness of the moments of $T$ is implied by the
finiteness of corresponding moments for $Y_{0}, Y'_{0}$, and $Y_{1}$. 

\begin{thm}[Theorem 4.2 from \cite{lindvall2002lectures}]
\label{lindvall}
Let $S_{n}$ and $S_{n}'$ be the renewal processes as above. If there
exists $\beta > 1$ such that $\mathbb{E}[Y_{0}^{\beta}]$,
$\mathbb{E}[Y_{0}^{' \beta}]$, and $\mathbb{E}[Y_{1}^{\beta}]$ are all
finite, then there exists a constant $C$ such that
$$
\mathbb{E}[T^{\beta}] \leq C ( \mathbb{E}[Y_{0}^{\beta}] +
\mathbb{E}[Y_{0}^{' \beta}]) < \infty \,.
$$
\end{thm}

Theorem \ref{lindvall} implies the following immediately:

\begin{thm}
\label{cor33}
Let $\Psi_{n}$ be a Markov chain on $(X, \mathcal{B})$
with transition kernel $\mathcal{P}$. Suppose $\Psi_n$ has an atom $\mathfrak{C}_{1}$ that is accessible and whose
return times are aperiodic. Let $\mu$ and $\nu$ be two probability distributions
on $X$, and assume that for some $\beta > 1$, 
$$
\mathbb E_\mu[\tau^\beta_{\mathfrak{C}_{1}}], \ \ \mathbb E_\nu[\tau^\beta_{\mathfrak{C}_{1}}]
\ \ \mbox{ and } \ \ \mathbb E_{\mathfrak{C}_{1}}[\tau^\beta_{\mathfrak{C}_{1}}] \ < \infty.
$$
Then 
$$
\lim_{n\to \infty}  n^\beta \| \mu \mathcal{P}^n - \nu \mathcal{P}^n\|_{\rm TV} = 0\ .
$$
\end{thm}

 The speed of convergence follows immediately by applying Theorem
 \ref{cor33} to the split chain $\tilde{\Psi}_n$.

\medskip

{\bf (C) Convergence rate for general Markov chain.}

It remains to pass the result of $\tilde{\Psi}_{n}$ to
$\Psi_{n}$. Note that if Theorem \ref{cor33} holds for $\tilde{\Psi}_{n}$, we have
$$
  \lim_{n\to \infty}  n^\beta \| \mu \mathcal{P}^n - \nu
  \mathcal{P}^n\|_{\rm TV} \leq \lim_{n\to \infty}  n^\beta \| \mu^{*}
  \mathcal{\tilde{P}}^n - \nu^{*} \mathcal{\tilde{P}}^n\|_{\rm TV} = 0 \,. 
$$
Therefore, result for $\Psi_{n}$ follows from the following lemma that
passes bounds of $\tau^{\beta}_{\mathfrak{C}}$ to bounds of
$\tau_{\mathfrak{C}_{1}}^{\beta}$. 

\begin{lem}[Lemma 3.1 of \cite{nummelin1983rate}]
\label{lem31}
Let $\Psi_{n}$ be an aperiodic Markov chain on $(X, \mathcal{B})$. If
$\mathfrak{C} \in \mathcal{B}$ is an accessible uniform reference set
and 
$$
  \sup_{x \in \mathfrak{C}} \mathbb{E}_{x}[
  \tau_{\mathfrak{C}}^{\beta}] < \infty
$$
for some $\beta > 0$, then for any probability measure $\mu$ such
that $\mathbb{E}_{\mu}[ \tau_{\mathfrak{C}}^{\beta}] < \infty$, we
have $\mathbb{E}_{\mu^{*}}[\tau_{\mathfrak{C}_{1}}^{\beta}] < C
\mathbb{E}_{\mu}[\tau_{\mathfrak{C}}^{\beta}] < \infty$ for some
constant $C$.  
\end{lem}

In summary, the rigorous result for polynomial rate of convergence is
as follows.

\begin{thm}[Theorem 2.7 of \cite{nummelin1983rate}]
\label{rigorous}
Let $\Psi_{n}$ be an aperiodic Markov chain on $(X, \mathcal{B})$ with
transition kernel $\mathcal{P}$. Assume $\Psi_{n}$ admits an accessible uniform
reference set $\mathfrak{C}$ such that
$$
  \sup_{x \in \mathfrak{C}} \mathbb{E}_{x}[
  \tau_{\mathfrak{C}}^{\beta}] < \infty
$$
for some $\beta > 0$, then for any probability measures $\mu$, $\nu$
on $X$ that satisfy 
$$
  \mathbb{E}_{\mu}[ \tau_{\mathfrak{C}}^{\beta}] < \infty, \quad
  \mbox{ and } \quad  \mathbb{E}_{\nu} [ \tau_{\mathfrak{C}}^{\beta}] <
  \infty \,,
$$
we have
$$
  \lim_{n\rightarrow \infty} n^{\beta} \| \mu \mathcal{P}^{n} - \nu
  \mathcal{P}^{n} \|_{TV} = 0 \,.
$$
\end{thm}

Using more precise bounds in Theorem \ref{lindvall} and Lemma
\ref{lem31} that involve the initial distributions, we can make the
following estimate that will be used to show the rate of correlation decay.  

\begin{cor}
\label{rigorous2}
Let $\Psi_{n}$ and $\mathfrak{C}$ be as in Theorem \ref{rigorous}. Then for any probability measures $\mu$, $\nu$
on $X$ that satisfy 
$$
  \mathbb{E}_{\mu}[ \tau_{\mathfrak{C}}^{\beta}] < \infty, \quad
  \mbox{ and } \quad  \mathbb{E}_{\nu} [ \tau_{\mathfrak{C}}^{\beta}] <
  \infty \,,
$$
there exists a constant $C$ such that
$$
  \sup_{n} n^{\beta} \| \mu \mathcal{P}^{n} - \nu
  \mathcal{P}^{n} \|_{TV} \leq C (   \mathbb{E}_{\mu}[
  \tau_{\mathfrak{C}}^{\beta}]  +  \mathbb{E}_{\nu}[ \tau_{\mathfrak{C}}^{\beta}] )\,.
$$
\end{cor}
\begin{proof}
Let $T$ be the coupling time of $\tilde{\Psi}_{n}$. 
For any $n \geq 0$,
$$
  \mathbb{E}_{\mu^{*}, \nu^{*}}[ T^{\beta}] \geq \sum_{k =
    n}^{\infty}k^{\beta} \mathbb{P}_{\mu^{*}, \nu^{*}}[ T = k] \geq
  n^\mathbb{\beta}\mathbb{P}_{\mu^{*}, \nu^{*} }[ T > n] \,.
$$

Hence have
$$
  \sup_{n} n^{\beta}\| \mu \mathcal{P}^{n} - \nu
  \mathcal{P}^{n} \|_{TV} \leq 2 \sup_{n} n^{\beta}
  \mathbb{P}_{\mu^{*}, \nu^{*}}[ T > n] \leq \mathbb{E}_{\mu^{*},
    \nu^{*}}[ T^{\beta}] \,.
$$

The corollary follows easily from more precise bounds in Theorem \ref{lindvall} and Lemma \ref{lem31}.

\end{proof}

\section{Verificable conditions for slow convergence}
The aim of this section is to convert conditions in Theorem
\ref{rigorous} to sufficient conditions that are verificable either numerically or analytically. 

Firstly, we will list sufficient conditions that will result in polynomial speed of convergence of $ \| \mu
\mathcal{P}^{n} - \nu \mathcal{P}^{n} \|_{TV}$ from Theorem \ref{rigorous}.
\begin{itemize}
\item[(1)] $\Psi_{n}$ is irreducible with respect to a non-trivial
  probability measure $\phi$.
\item[(2)] $\Psi_{n}$ is aperiodic.
\item[(3)] $\Psi_{n}$ admits a uniform reference set $\mathfrak{C}$ such
  that
$$
  \mathcal{P}(x, \cdot) \geq \eta \theta(\cdot) \quad \mbox{ for all }
  x \in \mathfrak{C}.
$$
  \item[(4)] There exists a constant $\beta > 0$ such that $\mathbb{E}_{\mu}[ \tau_{\mathfrak{C}}^{\beta}]< \infty$ and
    $\mathbb{E}_{\nu}[ \tau_{\mathfrak{C}}^{\beta}] < \infty$.
\item[(5)]  $
  \sup_{x\in \mathfrak{C}} \mathbb{E}_{x}[ \tau_{\mathfrak{C}}^{\beta}] < \infty
$ for the constant $\beta$ in (4). 

\end{itemize}

If in addition, we would like to show the existence of an invariant
probability measure $\pi$ and the polynomial speed of convergence
towards $\pi$, the following two more conditions are needed.

\begin{itemize}
  \item[(6)] $ \sup_{x\in \mathfrak{C}} \mathbb{E}_{x}[
    \tau_{\mathfrak{C}}] < \infty$.
\item[(7)]  $\mathbb{E}_{\pi}[ \tau_{\mathfrak{C}}^{\beta}] < \infty$
  for the constant $\beta$ in (4).
\end{itemize}

{\bf (A) Conditions that are verifiable analytically.} For most models, conditions (1)-(3) are relatively easy to check
analytically. Condition (1), i.e. the irreducibility, usually can be proved by constructing
an event with positive probability such that a positive-measured set is
``reachable''. In addition, it is well known that if the Markov
process has a continuous component, then the reachability of one point
implies the irreducibility. 

For many Markov processes, condition (3) is also easy to
check. Essentially all we need to show is that the probability measure
$\mathcal{P}(x, \cdot)$ for $x \in \mathfrak{C}$ has some uniform lower bound,
which is usually easy to prove by constructing events with positive
probability. It remains to show condition (2), i.e., the aperiodicity. In fact, if
$\mathfrak{C}$ is a uniform reference set with $\nu( \mathfrak{C}) >
0$, then $\Psi_{n}$ is a strongly aperiodic chain, which is obviously
aperiodic. 

We will give two examples about verifying these analytical conditions
in Section 5 and 6, both of which are Markov jump processes. We choose
to address the numerical verification of the ergodicity of stochastic differential
equations in a separate paper. Conditions (1)-(3) for stochastic
differential equations is usually linked to the H\"ormander's
condition \cite{malliavin1978stochastic, hormander1967hypoelliptic,
  hairer2011malliavin}. However, H\"ormander's condition alone does
not automatically imply condition (3) for a time-$h$ sample chain of
the stochastic differential equation. Some nontrivial work needs to be
done to verify this condition \cite{herzog2015practical}.

{\bf (B) Conditions that are easier to check numerically.} It is not difficult to show that a set is a uniform reference
set. However, a rigorous estimation of return times to a uniform
reference set is usually non-trivial. Most proofs rely on the
careful construction of a Lyapunov function $V$. It is well known that the
first passage time to the ``bottom'' of the Lyapunov function can be
estimated by calculating the ``drift''
$$
  \mathcal{P}V(x)  - V(x) = \int_{X} \mathcal{P}(x, \mathrm{d}y)V(y)
  = V(x)\,.
$$
Unfortunately, there is no universal approach to construct a Lyapunov
function for a Markov process. It may also be nontrivial to prove that a
given function is actually a Lyapunov function. We refer \cite{li2016polynomial}
for the examples of estimating first passage time by the Lyapunov
function method. 

On the other hand, the numerical computation of first passage times is
usually efficient and accurate. Therefore we choose to check
conditions (4)-(7) above numerically when a rigorous proof is out of
reach. We do not compute moments directly, because the moments of return times, i.e., $\tau^{\beta}$, usualy do not have finite variances. As a result, a large number of samples is usually necessary to
stabilize the estimate of $\mathbb{E}[ \tau^{\beta}]$. When the expectation of $\tau^{\beta}$ is close to blow-up,
numerically verifying whether it is finite becomes even more
difficult. Therefore instead, we observe that the finiteness of moments of a
random variable is closely related to its tail. 

Let $Z$ be a random variable that takes non-negative integer
values. Assume $\beta > 1$. The following two lemmas are straightforward.

\begin{lem} 
\label{m2t}
If $\mathbb{E}[Z^\beta]<\infty$, then 
$$
\lim_{n\to \infty} n^\beta \mathbb{P}[Z>n] = 0 \,.
$$
\end{lem}
\begin{proof} 
We skip the proof as this is a standard textbook result.
\end{proof}

\begin{lem}
\label{t2m}
If
$$
\mathbb{P}[Z > n] \le C n^{-\beta} \quad n > 0\,,
$$
then for any $\epsilon > 0$, we have
$$ 
\mathbb{E}[Z^{\beta - \epsilon}]< C_{1}C \epsilon^{-1} \,. 
$$
for some constant $C_{1}$ that depends on $\beta$. 
\end{lem}
\begin{proof} 
For any $\epsilon > 0$, we have
\begin{eqnarray*}
\mathbb{E}[Z^{\beta - \epsilon}] &
=  &\sum_{n = 0}^{\infty} n^{\beta - \epsilon}\mathbb{P}[Z = n] \\
\le C(\beta) \sum_{n = 0}^{\infty} \sum_{m = 0}^{n} m^{\beta - 1 -\epsilon} 
\mathbb{P}[Z = n]\\
& \le &  C(\beta) \sum_{n = 0}^{\infty} (n + 1)^{\beta - 1 -\epsilon} \mathbb{P}[Z > n] \\
& \le & C(\beta) \sum_{n = 0}^{\infty}
(n + 1)^{\beta - 1 -\epsilon} \min\{1,  Cn^{-\beta} \}\\
&<& C_{1}(\beta) C \epsilon^{-1}  \,,
\end{eqnarray*}
where $C(\beta)$ is a constant depending on $\beta$. 
\end{proof}

Therefore, condition (4) can be checked by computing the tails of
$\tau_{\mathfrak{C}}$. The polynomial tail of $\tau_{\mathfrak{C}}$
can be computed by measuring the slope of $\mathbb{P}[
\tau_{\mathfrak{C}} > n] $ vs. $n$ in the log-log plot. If
$\mathbb{P}[ \tau_{\mathfrak{C}} > n] \sim n^{-\beta}$ for some $\beta
> 1$, then by Lemma \ref{t2m}, we have
$\mathbb{E}[\tau_{\mathfrak{C}}^{ \beta - \epsilon}] < \infty$ for any
small $\epsilon > 0$, which implies condition (4) for parameter $\beta - \epsilon$. Condition (7)
can be verified in the same way by taking the numerical steady-state
distribution as the initial condition. 

It remains to check conditions (5) and (6). When $\beta > 1$, it is
easy to see that  (6) is implied by (5) immediately. Hence we focus on
the numerical verification of condition (5). As discussed before,
a numerical verification of the finiteness of
$\mathbb{E}_{x}[\tau^{\beta}_{\mathfrak{C}}]$ is difficult when
$\tau^{\beta}_{\mathfrak{C}}$ is close to blow-up. Instead, the
asymptotic property of the tail of $\tau_{\mathfrak{C}}$ is much more computable. Define 
$$
  \gamma (x) =  \sup_{n \geq 1}\frac{\mathbb{P}_{x}[ \tau_{\mathfrak{C}}
    > n]}{n^{-\beta}} \,.
$$
By Lemma \ref{t2m}, it is sufficient to numerically check the boundedness of
$\gamma (x)$ on $\mathfrak{C}$. Different from the moments, the tail
distribution $\mathbb{P}_{x}[ \tau_{\mathfrak{C}} > n]$ can be
computed efficiently with Monte Carlo simulation. In practice, we adopt
the following algorithm to check condition (5). This algorithm
provides reliable result in two examples in Section 5 and 6.
\begin{itemize}
  \item[(a)] Choose one point $x_{0}$ in $\mathfrak{C}$ and plot
    $\mathbb{P}_{x_{0}}[ \tau_{\mathfrak{C}} > n]$ vs. $n$ in the
    log-log plot. If the $\mathbb{P}_{x_{0}}[ \tau_{\mathfrak{C}} >
    n]$ vs. $n$ plot shows a straight line for large $n$, $\beta$ is
    determined by measuring the slope of this straight line. (If the
    plot does not form a straight line in the log-log plot, $\mathbb{P}_{x_{0}}[ \tau_{\mathfrak{C}} > n]$ does not have a
    polynomial tail.)
\item[(b)] For each $x \in \mathfrak{C}$, use the $\beta$ above to compute 
$$
  \gamma_{N}(x) := \sup_{1 \leq n \leq N}\frac{\mathbb{P}_{x}[ \tau_{\mathfrak{C}}
    > n]}{n^{-\beta}} 
$$
for increasing $N$ until $\gamma_{N}(x)$ is stabilized. Approximate
$\gamma (x)$ by $\gamma_{N}(x)$. 
\item[(c)] Numerically check that $\gamma (x)$ is uniformly bounded on
  $\mathfrak{C}$. This can be done by searching a grid of lattice
  points in $\mathfrak{C}$, finding monotonicity of $\gamma(x)$, or
  using some gradient-free numerical optimization algorithms \cite{anderson2001direct,
    walton2011modified, rios2013derivative, conn2009introduction}.
  (The simulation of derivatives of $\gamma (x)$ is usually
  not reliable.) 
\item[(d)] If during the optimization, $\gamma_{N}(x)$ can not be
  stabilized at some $x_{*} \in \mathfrak{C}$, repeat step $(a)$ for
  $x_{*}$ to update $\beta$. 
\end{itemize}

In the case of that the transition kernel of $\Psi_{n}$ is explicitly known, an
alternative approach of checking condition (5) is to compare the
transition kernel starting from each point in $\mathfrak{C}$.

\begin{pro}
\label{dominate}
Assume there exist $x_{*} \in \mathfrak{C}$ and constant $\delta >0$ such
that 
$$
  \mathcal{P}( x_{*}, \cdot) \geq \delta \mathcal{P}(x, \cdot)
$$
for all $x \in \mathfrak{C}$. Then
$\mathbb{E}_{x_{*}}[\tau_{\mathfrak{C}}^{\beta}] < \infty$ implies
$\sup_{x} \mathbb{E}_{x}[ \tau_{\mathfrak{C}}^{\beta}] < \infty$.
\end{pro}
\begin{proof}
For any $x \in \mathfrak{C}$, we have
\begin{eqnarray*}
\mathbb{E}_{x}[ \tau_{\mathfrak{C}}^{\beta}] & = & \mathcal{P}(x, \mathfrak{C}) + \int_{X
  \setminus \mathfrak{C}} \mathbb{E}_{y}[ (1 +
\tau_{\mathfrak{C}})^{\beta}] \mathcal{P}(x, \mathrm{d} y) \\
&\leq& \frac{1}{\delta} \mathcal{P}(x_{*}, \mathfrak{C}) + \frac{1}{\delta}\int_{X \setminus \mathfrak{C}}
\mathbb{E}_{y}[ (1 +\tau_{\mathfrak{C}})^{\beta}] \mathcal{P}(x_{*},
\mathrm{d} y) \\
&= & \frac{1}{\delta} \mathbb{E}_{x_{*}}[ \tau_{\mathfrak{C}}^{\beta}] \,.
\end{eqnarray*}
Hence $\mathbb{E}_{x}[ \tau_{\mathfrak{C}}^{\beta}]$ is uniformly
bounded if we have $\mathbb{E}_{x_{*}}[\tau_{\mathfrak{C}}^{\beta}] <
\infty$. 
\end{proof}

\bigskip

In summary, we are interested in the verification of the following four
conditions:

\begin{itemize}
  \item[{\bf (A1)}]  $\Psi_{n}$ is irreducible with respect to a non-trivial
  probability measure $\phi$.
\item[{\bf (A2)}] $\Psi_{n}$ admits a uniform reference set $\mathfrak{C}$ such
  that
$$
  \mathcal{P}(x, \cdot) \geq \eta \theta(\cdot) \quad \mbox{ for all }
  x \in \mathfrak{C} 
$$
and
$$
  \theta( \mathfrak{C}) > 0 \,.
$$
\item[{\bf (N1)}] Distributions
  $\mathbb{P}_{\mu}[ \tau_{\mathfrak{C}} \geq n]$ and
  $\mathbb{P}_{\pi}[ \tau_{\mathfrak{C}} \geq n]$ have polynomial
  tails $\sim n^{-\beta}$ for some $\beta  > 1$, where
  $\mu$ is the initial distribution that we are interested in, and
  $\pi$ is the invariant measure. If $\pi$ can not be given explicitly, we generate a numerical invariant
probability measure $\hat{\pi}$ instead.
\item[{\bf (N2)}] Function 
$$
  \gamma (x) = \sup_{ n\geq 1} \frac{\mathbb{P}_{x}[
    \tau_{\mathfrak{C}} > n]}{n^{-\beta}}
$$
is uniformly bounded on $\mathfrak{C}$. 
\end{itemize}

Condition {\bf (N2)} can be replaced by either of the following two
conditions. 
\begin{itemize}
\item[{\bf (N2)'}] There exist $x_{*} \in \mathfrak{C}$ and constant
  $\delta > 0$ such that
$$
  \mathcal{P}(x_{*}, \cdot) \geq \delta \mathcal{P}(x, \cdot) \,.
$$
In addition, the distribution $\mathbb{P}_{x^{*}}[\tau_{\mathfrak{C}}
> n]$ has a polynomial tail $\sim n^{-\beta}$ for some $\beta > 0$. 
  \item[{\bf (N2)''}] Function
$$
  h(x) := \mathbb{E}_{x}[ \tau_{\mathfrak{C}}^{\beta}]
$$
is continuous with respect to $x \in \mathfrak{C}$. $\mathfrak{C}$ is
a compact set.
\end{itemize}

\begin{rem}
Condition {\bf (N2)''} is usually applicable to stochastic
differential equations. If $\Phi_{t}$ is a stochastic differential
equation, it is well known that the integer moment of
first passage times of $\Phi_{t}$ can be obtained by solving a series
of Fokker-Planck-type equations successively \cite{risken1989fokker}. The continuous
dependency of moments with respect to the initial condition is not
hard to prove in this case. We will address the issue of speed of
convergence to steady states for stochastic differential equations in
a separate paper.

\end{rem}

\begin{rem}
When simulating $\mathbb{P}[\tau_{\mathfrak{C}} > n] $ for large $n$,
one needs to make sure that the numerical result is reliable. We apply the
Agresti-Coull interval \cite{agresti1998approximate} to determine the confidence
interval of $\mathbb{P}[\tau_{\mathfrak{C}} > n] $. Assume that in $N$ samples we
observed $m$ return times that are longer than $n$. Then define 
$$
  \tilde{N}  = N + z^{2}, \quad \tilde{p} = \frac{1}{\tilde{N}}(m +
  \frac{1}{2}z^{2}) \,,
$$
where $z = 1.96$ (the 0.975 quantile of a standard normal
distribution). The confidence interval for
$\mathbb{P}[\tau_{\mathfrak{C}} > n] $ is given by 
$$
  \tilde{p} \pm z \sqrt{\frac{1}{\tilde{N}} \tilde{p} (1 - \tilde{p})} \,.
$$
In order to make the result reliable, the confidence interval has to
be significantly smaller than $\tilde{p}$ itself. In other words $m$
cannot be too small. When $N$ is large, it is easy to see that the
ratio of $\sqrt{\frac{1}{\tilde{N}} \tilde{p} (1 - \tilde{p})} $ to
$\tilde{p}$ is roughly $m^{-1/2}$. In our simulation, the criterion is
that $\tilde{p}$ is reliable when $m$ is greater than $100$. 

\end{rem}

\begin{rem}
The rigorous result does not guarantee that the speed of convergence
equals the speed of contraction of the Markov operator. It is possible
that there exists a $\beta_{1}$ such that {\bf (N2)} holds for
$\beta_{1}$ and there exist initial distributions $\mu$, $\nu$ such that 
$$
  \mathbb{E}_{\mu}[\tau_{\mathfrak{C}}^{\beta_{1}}] < \infty, \quad
  \mbox{ and } \quad \mathbb{E}_{\nu}[\tau_{\mathfrak{C}}^{\beta_{1}}]
  < \infty \,,
$$
but $\mathbb{E}_{\pi}[ \tau_{\mathfrak{C}}^{\beta_{1}}] = \infty$. In
fact, in this situation the renewal theory only implies
$\mathbb{E}_{\pi}[ \tau_{\mathfrak{C}}^{\beta_{1} - 1}] < \infty$. In
other words, the speed of contraction of the Markov operator may be
faster than the speed of convergence to the invariant probability
measure. When the invariant probability
measure can not be explicitly given, we need to generate an invariant probability measure
numerically. Theoretically, the return time from the numerical invariant probability measure
could be different.

\end{rem}

\section{Main conclusion from conditions}

This section discusses the main conclusions one can obtain from the
sufficient conditions summarized in the previous section. What one can
learn from a direct numerical simulation is simple, as it only shows the speed of convergence 
with respect to one initial distribution. Different from that, our
method can support more general results as it incorporates both
numerical simulations and analytical proof. 

Applying Lemma \ref{t2m} and Theorem \ref{rigorous}, we can obtain the following direct consequences from {\bf (A1), (A2), (N1)} and
{\bf (N2)}.
\begin{itemize}
  \item There exists an invariant probability measure $\pi$.
\item For the initial distribution $\mu$ we have tested,
$$
 \lim_{n \rightarrow \infty} n^{\beta - \epsilon} \| \mu \mathcal{P}^{n} - \pi
 \|_{TV} = 0 \,.
$$
(or
$$
  \lim_{n \rightarrow \infty} n^{\beta - \epsilon} \| \mu \mathcal{P}^{n} - \hat{\pi}P^{n}
 \|_{TV} = 0 
$$
if $\hat{\pi}$ is numerically obtained) for any $\epsilon > 0$. 
\end{itemize}
In this section, we will try to go beyond that.

\subsection{Initial distributions.}
First we will show that the sufficient condition in the previous
subsection implies the polynomial convergence rate for a wider
class of initial conditions. This follows immediately from the following proposition.

\begin{pro}
\label{pro:class}
Assume $\Psi_{n}$ is irreducible with respect to $\phi$. If 
$$
  \sup_{x \in \mathfrak{C} } \mathbb{E}_{x}[ \tau_{\mathfrak{C}}^{\beta}] < \infty \,,
$$
then $\mathbb{E}_{x}[ \tau_{\mathfrak{C}}^{\beta}]  < \infty$ for
$\phi$-almost every $x \in X$.
\end{pro}
\begin{proof}
Let $A \in \mathcal{B}$ be the set such that $\forall x \in A$,
$\mathbb{E}_x [\tau_{\mathfrak{C}}^{\beta}]=\infty$. \\ 
Suppose that $\phi(A)>0$. Since $\Psi_{n}$ is irreducible,
$\sum_{n=1}^\infty P^n(x,A)>0$ for all $x \in \mathfrak{C}$. Choose
$x_0 \in \mathfrak{C}$, by irreducibility there exists $n>0$ such that
$P^n(x_0,A)>0$. 

Define
the n-step taboo transition probability by 
$$
{_A}P^n(n,B):=P_x(\Phi_x \in B,
\tau_A \geq n) \,,
$$
where $x \in X, B \in \mathcal{B}$, to be the
probability of a transition to $B$ in n steps of the chain, avoiding
the set $A$. By the following last exit decomposition, we have 
$$
P^n(x_0,A) = \sum_{k=1}^n \int_\mathfrak{C} P^k (x_0, dy) _\mathfrak{C}P^{n-k}(y, A) >0 \,.
$$
Therefore, there exists $x_{1} \in \mathfrak{C}$  and $m>0$ such that
${_\mathfrak{C}}P^m(x_1,A)>0$. This implies
\begin{eqnarray*}
\mathbb{E}_{x_1} [\tau_{\mathfrak{C} ^\beta}]
& = & \sum_{n=0}^\infty n^\beta  \cdot   {}_\mathfrak{C}P^n(x_1 , \mathfrak{C}) \\
& \geq & \sum_{n=m} ^\infty n^\beta \int_A {}_\mathfrak{C}P^m(x_1,dw) \cdot {}_{\mathfrak{C}}P^{n-m}(w,\mathfrak{C})\\
&\geq& \int_A {}_\mathfrak{C}P^m(x_1,dw) \mathbb{E}_{w}[ \tau_{\mathfrak{C}}] \\
& = & \infty.
\end{eqnarray*}
But $\sup_{x \in \mathfrak{C}} \mathbb{E}_{x}[ \tau_{\mathfrak{C}}^{\beta}] <
\infty$ and $x_{1} \in \mathfrak{C}$ imply that $\mathbb{E}_{x_1} [\tau_{\mathfrak{C} ^\beta}] <\infty$, resulting in a contradiction. Therefore $\phi(A)=0$.
\end{proof}

By proposition \ref{pro:class}, if condition {\bf (N2)} holds, then
{\bf (N1)} automatically holds for
$\phi$-almost all $\mu = \delta_{x}$, $x \in X$.

\subsection{Decay of correlation}

Let $\xi(x)$ and $\eta(x)$ be two functions in $L^{\infty}(X)$. The
decay of correlation $C^{\xi, \eta}_{\mu}(n)$ is denoted by
$$
  C^{\xi, \eta}_{\mu}(n) := |\int (\mathcal{P}^{n} \eta)( x) \xi(x)
  \mu(\mathrm{d}x) - \int (\mathcal{P}^{n} \eta)( x) \mu(\mathrm{d}x) \int
\xi(x) \mu( \mathrm{d}x) | \,.
$$
We have the following proposition regarding the decay of correlation.
\begin{pro}
\label{doc}
If {\bf (A1), (A2), (N1)} and {\bf (N2)} are satisfied, then 
$$
  C^{\xi, \eta}_{\mu}(n)  \leq o(n^{\epsilon-\beta} )
$$
for any $\epsilon > 0$.
\end{pro}
\begin{proof}
We have
\begin{eqnarray*}
& &  \left|\int (P^{n} \eta)(
  \mathbf{E}) \xi( \mathbf{E}) \mu( \mathrm{d}\mathbf{E}) - 
  \int (P^{n}\eta)( \mathbf{E}) \mu(\mathrm{d}\mathbf{E}) \int \xi( \mathbf{E}) \mu(
  \mathrm{d} \mathbf{E}) \right| \\
& = & \left| \int \xi (\mathbf{E}) \left( (P^{n}\eta)(\mathbf{E}) - 
\int (P^{n}\eta)(\mathbf{Z}) \mu(\mathrm{d}\mathbf{Z}) \right) 
\mu( \mathrm{d} \mathbf{E}) \right|\\
& \le & \|\xi\|_{L^{\infty}} \ \|\eta\|_{L^{\infty}} \ \int \|\delta_{\mathbf{E}}P^n - \mu P^n\|_{TV} \ 
\mu( \mathrm{d} \mathbf{E}) \,.
\end{eqnarray*}
Since {\bf (A1), (A2), (N1)} and
{\bf (N2)} hold, it follows from Corollary \ref{rigorous2} that
$$
  \|\delta_{\mathbf{E}}P^n - \mu P^n\|_{TV} \leq C (
  \mathbf{E}_{\mathbf{E}}[ \tau_{\mathfrak{C}}^{\beta-\epsilon/2}] +
  \mathbf{E}_{\mu}[ \tau_{\mathfrak{C}}^{\beta - \epsilon/2}]
  )n^{\epsilon/2-\beta}  \,.
$$
We have $\mathbf{E}_{\mu}[ \tau_{\mathfrak{C}}^{\beta- \epsilon/2}]  < \infty$ and 
$$
  \int  \mathbf{E}_{\mathbf{E}}[ \tau_{\mathfrak{C}}^{\beta - \epsilon/2}]  \mu(
  \mathrm{d} \mathbf{E} ) = \mathbf{E}_{\mu}[
  \tau_{\mathfrak{C}}^{\beta - \epsilon/2}] < \infty \,.
$$
Therefore, 
$$
  \|\delta_{\mathbf{E}}P^n - \mu P^n\|_{TV}  \leq O(1) \cdot
  n^{\epsilon/2 -\beta} = o( n^{\epsilon - \beta}) \,.
$$
This completes the proof.
\end{proof}

\subsection{Results for time-continuous process}

We have only showed sufficient conditions of polynomial convergence
rate for time-discrete Markov processes. If $\Psi_{t}$ is a
time-continuous Markov process, our method is also applicable after
some additional effort. Instead of investigating the infinitesimal generator of
$\Psi_{t}$, we will work on the time-$h$ sample chain of
$\Psi_{t}$, which is denoted by $\Psi_{n}^{h} := \Psi_{nh}$. Same as
the time-discrete case, we study the tail of return time to a uniform
reference set $\mathfrak{C}$.

To pass the existence of an invariant measure from $\Psi^{h}_{n}$ to $\Psi_{t}$, the
``continuity at zero'' is necessary. We will show that in the following propositions.

\begin{pro}
\label{c0}
Assume $\Psi^{h}_{n}$ is aperiodic and admits an invariant probability
measure $\pi_{h}$. If in addition, 
$$
  \| \pi_{h} \mathcal{P}^{\delta} - \pi_{h} \|_{TV}
  \rightarrow 0 \quad \mbox{ as } \delta \rightarrow 0 \,,
$$
then $\pi = \pi_{h}$ is invariant for any $\Psi_{t}$.
\end{pro}
\begin{proof}
Notice that $\pi_{h}$ is invariant for any $\Psi^{hj/k}_{n}$, where $j, k
\in \mathbb{Z}^{+}$ (Theorem 10.4.5 of \cite{meyn2009markov}). Then without loss of generality, assume $t/h \notin
\mathbb{Q}$. By the density of orbits in irrational rotations, there
exist sequences $a_{n}$, $b_{n} \in \mathbb{Z}^{+}$ such that
$$
  d_{n} := t - \frac{a_{n}}{b_{n}}h \rightarrow 0
$$
from right. Then 
$$
  \pi_{h} P^{t} = \pi_{h}P^{\frac{a_{n}}{b_{n}} h}
  P^{d_{n}} \,.
$$
Therefore,
$$
\| \pi_{h}P^{t} - \pi_{h} \|_{TV} \leq   \lim_{n \rightarrow \infty}
\| \pi_{h}P^{d_{n}} - \pi_{h} \|_{TV} = 0
$$
by the assumption of ``continuity at zero''. Hence
$\pi_{h}$ is invariant with respect to $P^{t}$. 
\end{proof}

\begin{pro}
Assume {\bf (A1), (A2) , (N1)} and {\bf (N2)} hold for
$\Psi_{n}^{h}$. If in addition 
$$
  \| \pi_{h} \mathcal{P}^{\delta} - \pi_{h} \|_{TV}
  \rightarrow 0 \quad \mbox{ as } \delta \rightarrow 0 \,,
$$
where $\pi_{h}$ is the invariant measure for
$\Psi^{h}_{n}$, then for any small $\epsilon
> 0$, 
$$
  \lim_{t\rightarrow \infty} t^{\beta - \epsilon} \| \mu P^{t} - \pi_{h}
  \|_{TV} = 0 \,.
$$
\end{pro}
\begin{proof}
It follows from Proposition \ref{c0} that $\pi_{h}$ is invariant with
respect to $P^{t}$ for any $t > 0$. In addition, 
$$
  \| \mu P^{t} - \pi_{h}\|_{TV} = \| \mu P^{nh}P^{r} - \pi_{h}\|_{TV}
  = \| ( \mu P^{nh} - \pi_{h})P^{r} \|_{TV} \leq \|\mu P^{nh} -
  \pi_{h} \|_{TV} \,,
$$
where $n$ is the greatest integer that is smaller than $t/h$. This
implies 
$$
  \lim_{t\rightarrow \infty} t^{\beta - \epsilon} \| \mu P^{t} - \pi_{h}
  \|_{TV} = 0 \,.
$$

\end{proof}

In many situations, especially when the time step has to be small, it
may be easier to simulate $\tau_{\mathfrak{C}}$ for $\Psi_{t}$ instead
of $\Psi^{h}_{n}$. Let $\hat{\tau}_{\mathfrak{C}}$ be the
first-passage time to $\mathfrak{C}$ for
$\Psi^{h}_{n}$. Let $h$ be a fixed parameter. The first return time
and first passage time for
$\Psi_{t}$ are defined as
$$
  \tau_{\mathfrak{C}} = \tau_{\mathfrak{C}}(h)  = \inf_{ t \geq h} \{ \Psi_{t} \in \mathfrak{C} \}
  \mbox{ and } \sigma_{\mathfrak{C}} =  \inf_{ t \geq 0} \{ \Psi_{t}
  \in \mathfrak{C} \} \,.
$$
We drop the notation $(h)$ when it does not lead to confusion. Here we have to treat return time and passage time differently because otherwise $\sigma_{\mathfrak{C}}$ is zero for all initial conditions
within $\mathfrak{C}$. We say  {\bf (N1)} and {\bf (N2)} hold for
$\Psi_{t}$ with parameters $h$ and $\beta$ if 
\begin{itemize}
  \item[{\bf (N1)}] Distributions $\mathbb{P}_{\mu}[
    \tau_{\mathfrak{C}}(h) \geq t]$ and  $\mathbb{P}_{\pi}[
    \tau_{\mathfrak{C}}(h) \geq t]$ have polynomial tails $\sim
    t^{-\beta}$ for some $\beta > 1$.
\item[{\bf (N2)}] Function
$$
  \gamma (x) = \sup_{ t \geq h}\frac{\mathbb{P}_{x}[
    \tau_{\mathfrak{C}}(h) > t]}{t^{-\beta}} 
$$
is uniformly bounded on $\mathfrak{C}$.  
\end{itemize}

The following theorem gives the relation
between tail of $\hat{\tau}_{\mathfrak{C}}$ and tail of $\tau_{\mathfrak{C}}$.

\begin{thm}
\label{c2d}
Assume {\bf (N1)} and {\bf (N2)} hold for $\Psi_{t}$ with parameter
$h$ and $\beta$. If further 
$$
  \inf_{x \in \mathfrak{C}}\mathbb{P}_{x}[ \Psi_{h} = \Psi_{0}  ] >
  \gamma > 0 
$$
for the step size $h$, then for any $\epsilon > 0$,  {\bf (N1)}
and {\bf (N2)}  hold for $\Psi^{h}_{n}$ with parameter $\beta - \epsilon$.
\end{thm}
\begin{proof}
Define the following stopping times $\tau_{n}$ and random times
$\xi_{n}$. $\xi_{0} = 1$. $\tau_{n} = \inf_{t \geq \xi_{n-1}h}
\{ \Psi_{t} \in \mathfrak{C} \}$. $\xi_{n} = \left \lceil \tau_{n}/h
  \right \rceil$. It is possible that $\Psi_{t}$ leaves $\mathfrak{C}$
  after $\tau_{n}$ but before $\xi_{n}$ and thus makes a ``false return''. To verify  {\bf (N1)}
and {\bf (N2)}  for $\Psi^{h}_{n}$, we need to estimate the number of
``false returns''.

Let $N$ be the number of ``false returns'' of $\Psi_{t}$:
$$
  N =  \inf_{n > 0} \{\Psi_{\xi_{n} h} \in \mathfrak{C} \} \,.
$$
It is easy to see that $\hat{\tau}_{\mathfrak{C}} = \xi
_{N}$. Therefore, for any small $\sigma > 0$, we have
$$
  \{ \xi_{N} > n^{1 + \delta} \} \subset \{N > n^{\delta} \} \cup
  \bigcup_{k = 0}^{\left \lfloor n^{\delta} \right \rfloor} \{
  \xi_{k+1} - \xi_{k} > n, N > k \} \,.
$$

Without loss of
generality we assume $n > 1$ in the tail estimates throughout the proof. For any $\Psi_{0}$, by the Markov property, we have
$$
  \mathbb{P}_{\Psi_{0}}[N = n | N > n-1] \geq \mathbb{P}_{\Psi_{\tau_{n}}}[
  \Psi_{\xi_{n}h} = \Psi_{\tau_{n}}] .
$$
Since $\Psi_{\tau_{n}} \in \mathfrak{C}$, this probability is at least
$\delta$. Therefore, we have
$$
  \mathbb{P}_{\Psi_{0}}[ N > n] \leq (1 - \gamma)^n 
$$
for any $\Psi_{0}$ and $n > 1$. 

Let $\epsilon > 0$ be an arbitrary small number. For $k \geq 1$, we have
\begin{eqnarray*}
\mathbb{P}_{\Psi_{0}}[ \xi_{k+1} - \xi_{k} > n , N > k] & \leq  &
\mathbb{P}_{\Psi_{0}}[\tau_{k+1} - \xi_{k}h > (n - 1)h , N > k]\\
&\leq& \mathbb{P}_{\Psi_{0}}[\tau_{k+1} - \xi_{k} h > (n-1)h ,
\Psi_{\xi_{k}h} \notin \mathfrak{C}]\\
&=& \mathbb{P}_{\Psi_{\xi_{k} h}}[ \sigma_{\mathfrak{C}}^{\beta - \epsilon/2} > (n -
1)^{\beta - \epsilon/2} h^{\beta - \epsilon/2} , \Psi_{\xi_{k}h} \notin \mathfrak{C}] \\
&\leq& \frac{ \mathbb{E}_{\Psi_{\xi_{k} h}}[\sigma_{\mathfrak{C}}^{\beta
  -\epsilon/2}
    \mathbf{1}_{\Psi_{\xi_{k}h} \in \mathfrak{C}}]}{(n - 1)^{\beta - \epsilon/2}
    h^{\beta - \epsilon/2}}  =  \frac{ \mathbb{E}_{\Psi_{\xi_{k}
        h}}[\sigma_{\mathfrak{C}}^{\beta - \epsilon/2}]}{(n - 1)^{\beta
      - \epsilon/2}
    h^{\beta - \epsilon/2}}  \,,
\end{eqnarray*}
where $\mathbb{P}_{\Psi_{\xi_{k}h}}$ and
$\mathbb{E}_{\Psi_{\xi_{k}h}}$ mean letting the initial distribution
to be $\Psi_{\xi_{k}h}$. The second to last inequality follows from Markov
inequality. Then it is
sufficient to show that $\mathbb{E}_{\Psi_{\xi_{k}
    h}}[\sigma_{\mathfrak{C}}^{\beta - \epsilon/2}] < \infty$. 

Define $\tau_{\mathfrak{C}, r} = \inf_{t > r}\{ \Psi_{t} \in
\mathfrak{C} \} - r$ be the first return time when starting from $\Psi_{r}$. Then 
\begin{eqnarray*}
 \mathbb{E}_{\Psi_{\xi_{k}
    h}}[\sigma_{\mathfrak{C}}^{\beta - \epsilon/2}] &  = & \int_{\mathfrak{C}}
\int_{0}^{h} \mathbb{E}_{x}[ \tau_{\mathfrak{C}, r}^{\beta - \epsilon/2}]
\mathbb{P}_{\Psi_{0}}[ \xi_{k} h - \tau_{k} = r,
\Psi_{\tau_{k}} = x] \mathrm{d}r \mathrm{d}x  \,.\\
\end{eqnarray*}
Notice that $r \leq h$. For each sample path starting from $\tau_{k}$, we have
$\tau_{\mathfrak{C}, r} \leq \tau_{\mathfrak{C}}(h)$. This implies
$$
  \mathbb{E}_{\Psi_{\xi_{k}
    h}}[\sigma_{\mathfrak{C}}^{\beta - \epsilon/2}] \leq \int_{\mathfrak{C}}
\int_{0}^{h} \mathbb{E}_{x}[ \tau_{\mathfrak{C}}^{\beta - \epsilon/2}]
\mathbb{P}_{\Psi_{0}}[ \xi_{k} h - \tau_{k} = r,
\Psi_{\tau_{k}} = x] \mathrm{d}r \mathrm{d}x \leq \sup_{x \in
  \mathfrak{C}} \mathbb{E}_{x}[ \tau_{\mathfrak{C}}(h)^{\beta -
  \epsilon/2}] \,.
$$ 
Since {\bf (N2)} holds for $\Psi_{t}$, by Lemma \ref{t2m}, we have
$$
  \sup_{x \in
  \mathfrak{C}} \mathbb{E}_{x}[ \tau_{\mathfrak{C}}(h)^{\beta -
  \epsilon/2}]  \leq C_{0}
$$
for some constant $C_{0}$. This implies
$$
  \mathbb{P}_{\Psi_{0}}[ \xi_{k+1} - \xi_{k} > n , N > k]  \leq C
  n^{-(\beta - \epsilon/2)}
$$
for some constant $C$. 

If $k = 0$, we need to estimate
$$
  \mathbb{P}_{\Psi_{0}}[ \tau_{\mathfrak{C}}(h) > (n-1)h] 
$$
for $\Psi_{0} = x$, $x \in \mathfrak{C}$. Since {\bf (N1)} and {\bf
(N2) }hold for $\Psi_{t}$, we have
$$
  \mathbb{P}_{x}[ \tau_{\mathfrak{C}}(h) > (n-1)h] \leq C'
  n^{-\beta} \leq C' n^{-(\beta - \epsilon/2)} \,,
$$
where $C'$ is independent of $x$. Therefore, we have
\begin{eqnarray*}
\mathbb{P}_{x}[\xi_{N} > n^{1 + \delta}] & \leq  & \mathbb{P}[ N >
n^{\delta}] + \sum_{k = 0}^{\left \lfloor n^{\delta} \right \rfloor}
\mathbb{P}_{x}[\xi_{k+1} - \xi_{k} > n, N > k] \\
&\leq& (1 - \gamma)^{n^{\delta}} + n^{\delta} \cdot \max\{C, C'\}
n^{-(\beta - \epsilon/2)} \,.
\end{eqnarray*}
Note that $ (1 - \gamma)^{n^{\delta}}$ converges to $0$ faster than
$n^{-\beta - \epsilon/2}$. For any small $\epsilon > 0$, by making $\delta$
sufficiently small, we have 
$$
  \mathbb{P}_{x}[ \hat{\tau}_{\mathfrak{C}} > n] \leq C_{1} n^{-(\beta - \epsilon)}
$$
for some constant $C_{1}$ that is independent of $x \in
\mathfrak{C}$. This verifies condition {\bf (N2)} for $\Phi^{h}_{n}$.

Similarly, if $\Psi_{0} \sim \mu$ or
$\Psi_{0} \sim \pi$, there exist constants $C'_{\mu}$ or $C'_{\pi}$
such that
$$
  \mathbb{P}_{\mu}[ \tau_{\mathfrak{C}}(h) > (n-1)h] \leq C'_{\mu}
  n^{-(\beta - \epsilon/2)}  
$$
or
$$
  \mathbb{P}_{\pi}[ \tau_{\mathfrak{C}}(h) > (n-1)h] \leq C'_{\pi}
  n^{-(\beta - \epsilon/2)}   \,.
$$
Same calculation as above verifies condition {\bf (N1)} for
$\Psi^{h}_{n}$. This completes the proof. 

\end{proof}

Finally, it is trivial to show that if {\bf (A1), (A2), (N1)} and
{\bf (N2)} hold for $\Psi^{h}_{n}$, then Proposition \ref{doc} remains
true for $\Psi_{t}$.

\subsection{Summary of conclusions.} In summary, assume {\bf (A1), (A2), (N1)} and
{\bf (N2)} hold for $\Psi_{n}$, then
\begin{itemize}
  \item[(a)] $\Psi_{n}$ admits an invariant probability measure $\pi$.
\item[(b)] Polynomial convergence rate to $\pi$:
$$
\lim_{n\rightarrow \infty}  n^{\beta - \epsilon} \| \mu \mathcal{P}^{n} - \pi
\|_{TV}  = 0 
$$
for any $\epsilon > 0$.
\item[(c)] Polynomial decay rate of correlation:
$$
  \lim_{n\rightarrow \infty}  n^{\beta - \epsilon} C_{\pi}^{\xi, \eta} = 0 
$$
for any $\epsilon > 0$.
  \item[(d)] Polynomial convergence rate to $\pi$. for any $\epsilon >
    0$, we have
$$
  \lim_{n\rightarrow \infty}  n^{\beta - \epsilon} \| \delta_{x} \mathcal{P}^{n} - \pi \|_{TV}  = 0
$$
for $\phi$-almost every $x \in X$. 
\end{itemize}
$\pi$ in (b) and (d) should be replaced by $\hat{\pi}P^{n}$ Pn if $\hat{\pi}$ is numerically generated.

If $\Psi_{t}$ is a time-continuous Markov process with transition
kernel $\mathcal{P}^{t}$. Assuming that the ``continuity at zero'' condition
in Proposition \ref{c0} is satisfied, if {\bf (A1), (A2), (N1)} and
{\bf (N2)} hold for $\Psi^{h}_{n}$, then conclusions (a)-(d)
also hold for $\Psi_{t}$. If {\bf (N1)} and
{\bf (N2)} hold for $\Psi_{t}$, then conclusions (a)-(d) still hold
for $\Psi_{t}$ as one can put $\epsilon/2$ into Theorem \ref{c2d}
and Theorem \ref{rigorous}.

\section{Example: stochastic energy exchange model}
\subsection{Derivation from deterministic dynamics} Consider a long
tube of gas that is connected to two thermalized boundaries. Assume that
all gas molecules are rigid moving particles and that the only interaction
between particles is rigid body collision. Then we have a
very complicated deterministic dynamical system like as described in
Figure \ref{tube}. Apparently this is a very difficult multi-body
problem. To reduce its significant difficulty, one approach is to
``localize'' all particles such that particles are trapped in
one-dimensional cells, like the one described in Figure
\ref{lcp}. This is called the locally confined particle system, in which energy transport still exists because neighboring  particles
can collide through the gates between cells. 

\begin{figure}[htbp]
\centerline{\includegraphics[width = \linewidth]{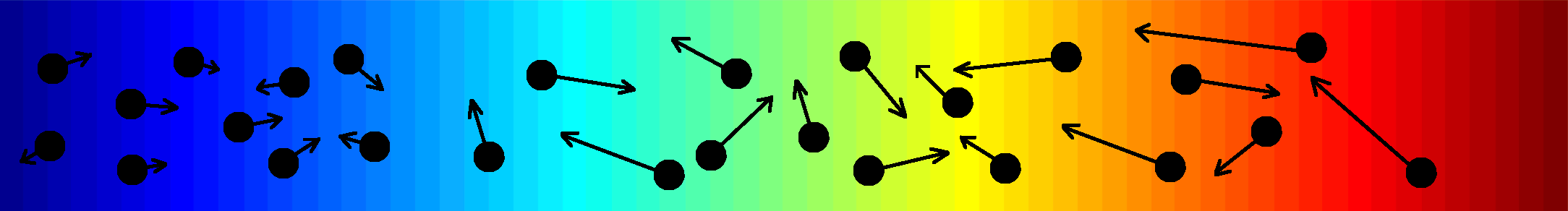}}
\caption[]{\label{tube} Moving particles in a long tube.}
\end{figure}

\begin{figure}[htbp]
\centerline{\includegraphics[width = \linewidth]{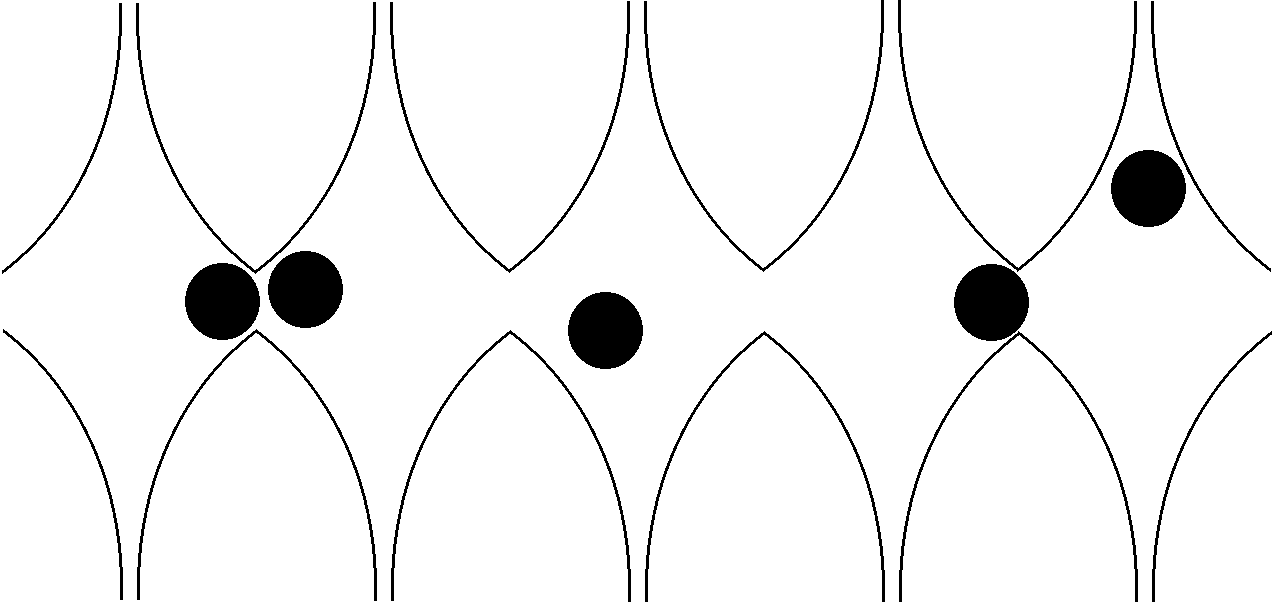}}
\caption[]{\label{lcp} Particles are confined in a chain of cells.}
\end{figure}

As probably the simplest deterministic microscopic heat conduction model, the locally confined
particle system is studied in various literatures. For example, under
certain assumptions, the ergodicity of this model has been proved in
\cite{bunimovich1992ergodic}. Further rigorous investigation of the locally confined
particle system is known to be difficult. Each cell in the locally
confined particle system is a chaotic billiard table. And chaotic
billiards are known to have many stochastic properties. Therefore, a
natural strategy is to use a Markov process to approximate the change
of particle energy in the locally confined particle system. 

In \cite{li2015stochastic}, numerical simulations of first particle-particle
collision in the locally confined particle system shows that for a
pair of adjacent particles with energy $(E_{1}, E_{2})$, the first
particle-particle collision time has an exponential tail with slope
$R(E_{1}, E_{2})$. In addition, when $\min \{E_{1}, E_{2}\} \ll 1$, we
have $ R(E_{1}, E_{2}) \sim \sqrt{ \min \{E_{1}, E_{2}\} }$. The
heuristic justification of this rate is that when a particle has low
energy, it has to move to the gate between cells by itself in order to
have the next energy exchange. Hence a very slow particle dominates the next particle-particle collision time. If in addition we assume
energy exchange in a particle-particle collision is done in a ``random
halves'' way, we have a stochastic energy exchange model described in
the next subsection.

\subsection{Model Description}

Consider a chain of $N$ linearly ordered lattice sites $\{1,2,\cdots, N\}$, each
storing a fixed amount of energy $e_i$, $i=1,2,\cdots, N$. The chain is
connected to two heat baths at the ends with temperatures $T_L$
and $T_R$, respectively.

An exponential clock is associated with each pair of adjacent sites,
with rate $R(e_i, e_{i+1})=\sqrt{\text{min}\{e_i, e_{i+1}\}}$, called
the stochastic energy exchange rate. Two clocks are associated with
the ends of the chain and the heat baths. The clock between the left
(resp. right) heat bath and the first (resp. last) site has a rate
$\sqrt{\min \{ T_{L}, e_{1}}\}$ (resp. $\sqrt{\min\{ T_{R}, e_{N}}\}$). 

When the $i$-th clock rings, sites $i$ and $i+1$ exchange energy as 
$$
(e_i',e_{i+1}')=(p(e_i+e_{i+1}), (1-p)(e_i+e_{i+1})) \,,
$$
where $p$ satisfies the uniform distribution on $(0,1)$ and is
independent of everything else. When
a clock involving heat bath rings, the corresponding site exchanges
energy with an exponential random variable with mean $T_L$ (or
$T_R$) in the same ``random halves'' fashion. This is to say, we have
$$
  e_{1}' = p( e_{1} + \rho_{L}), \quad \phi_{L} \sim Exp(T_{L}) 
$$
and
$$
  e_{N}' = p( e_{N} + \rho_{R}), \quad \rho_{R} \sim Exp(T_{R})  \,,
$$
where $p$ still satisfies the uniform distribution on $(0, 1)$ and is
independent of everything else. For the sake of simplicity, clocks at left and right boundaries are
denoted by clock $0$ and clock $N+1$, respectively.

It is easy to see that the stochastic energy exchange model generates
a Markov jump process $\mathbf{E}_{t}=(e_1(t), \cdots, e_N(t) )$ on
$\mathbb{R}^{N}_{+}$, where $e_{i}(t)$ represents the site
$i$ energy at time $t$. We denote the transition kernel of
$\mathbf{E}_{t}$ by $P^{t}( \mathbf{E}, \cdot)$ for $\mathbf{E} \in
\mathbb{R}^{N}_{+}$. 

Let $h > 0$ be a fixed number that represents the step size. The
time-$h$ sampling chain of $\mathbf{E}_{t}$ is denoted by
$\mathbf{E}^{h}_{n}$, or simply $\mathbf{E}_{n}$ when it does not lead to a
confusion. The transition kernel of $\mathbf{E}_{n}$ is denoted by
$P(\mathbf{E}, \cdot)$.

\subsection{Verifying {\bf (A1)} and {\bf (A2)}.}

We will first work on the time-$h$ chain $\mathbf{E}_{n}$. The
verification of analytical conditions of $\mathbf{E}_{n}$ is based on
the following Theorem.

\begin{thm}
\label{urs1}
For any set $K \subset \mathbb{R}^{N}_{+}$ of the form $K = \{(e_{1},
\cdots, e_{N} ) \,|\, 0 < c_{i} \leq e_{i} \leq C_{i} , i = 1 \sim N \}$ and any $h > 0$,
there exists a constant $\eta > 0$ such that
$$
  P(\mathbf{E}, \cdot) > \eta U_{K}(\cdot) \,,
$$
for any $\mathbf{E} \in K$, where $U_{K}$ is the uniform probability distribution over $K$.
\end{thm}
\begin{proof}
For $\mathbf{E}=\{e_1,\ldots, e_N\} \in K$ and $\mathrm{d}\mathbf{E}=\{ ( \mathrm{d}e_{1}, \cdots, \mathrm{d}e_{N}), de_i >0,
i=1 \sim N\}$. Assume $\mathrm{d}e_{i}$s are sufficiently small. Let 
$$
B(\mathbf{E},\mathrm{d}\mathbf{E}) = \{ (x_{1}, \cdots, x_{N}) \in \mathbb{R}^{N} \, | \, e_i \leq x_i \leq
e_i+\mathrm{d}e_i \})
$$ 
denote the hypercube in $\mathbb{R}^{N}$. It then suffices to prove
that for any $\mathbf{E}_{0} = \{ \bar{e}_{1}, \cdots, \bar{e}_{N}
\} \in K$, we have
\begin{equation}
\label{eqn:analytic}
P(\mathbf{E}_{0}, B(\mathbf{E},d\mathbf{E})) > \sigma \mathrm{d}e_{1}\mathrm{d}e_{2}\cdots
\mathrm{d}e_{N} \,,
\end{equation}
where $\sigma$ is a strictly positive constant that is independent of
$\mathbf{E}$ and $\mathbf{E}_{0}$.

We will then construct a sequence of events to go from the state
$\mathbf{E}_{0}$ to $B(\mathbf{E}, \mathrm{d}\mathbf{E})$ with desired positive
probability. Let $\delta = \frac{h}{2N+1}$ and let $\epsilon>0$ be
sufficiently small such that $\epsilon < \min \{c_i,i=1 \sim N\}$. Let
$H=\sum_{i=1} ^N e_i+\mathrm{d}e_i$. We consider the events $S_1 \cdots, S_N$
and $F_1, \ldots,F_{N+1}$, where $S_i$ and $F_j$ specifies what
happens on the time interval $(i\delta, (i+1)\delta]$ and
$( N\delta + (j-1)\delta, N \delta + j\delta]$, respectively.

\begin{itemize}
  \item $S_{i} =$ $\{e_i(i\delta) \in [\epsilon/2,\epsilon]\}$ and  \{
    the
    $i$-th clock rings exactly once, all other clocks are silent on
    $((i-1)\delta, i\delta]$ \}. 
    \item $F_{1} = $ {Energy emitted by right heat bath $\in (H, 2H)$
      } and { the
    $N$-th clock rings exactly once, all other clocks are silent on $(
    N\delta, (N+1) \delta]$ }.
\item $F_{j} = $ $\{e_j(N\delta + j \delta) \in [e_{N+2 - j},
  e_{N+2-j}+\mathrm{d}e_j]\}$ and $\{$ the
    $(N+1-j)$-th clock rings exactly once, all other clocks are silent on $(
    N\delta + (j-1) \delta, N\delta + j \delta]$ $\}$ for $j = 2, \cdots,
  N+1$. 
\end{itemize}

The idea is that the energy at each site is first transported to the right
heat bath, with only an amount of energy between $\epsilon/2$ and
$\epsilon$ left at each site. Then a sufficiently large amount of energy is injected into
the chain from the right heat bath so that it is always possible for
site $j$ to acquire an amount of energy between $e_j$ and $e_j+\mathrm{d}e_j$
by passing the rest to site $j-1$, where sites $0$ and $N+1$ denote
the left and right heat baths respectively.  

It is easy to show that the probability of occurence of the sequence of
events described above is always strictly positive. Here is a brief
list of considerations. We will leave detailed calculations to the reader.  
\begin{enumerate}
\item[(a)] At each clock tick, we can give a lower bound on the rate
  of the $i^{th}$ clock (i.e., $\sqrt{\text{min}\{e_i,e_{i+1}\}}$), since
  $0 < \epsilon/2 \leq e_i$ for all $i =1 \sim N$ throughout the
  entire event. 
\item[(b)] Let $p \in (0,1)$ be the fraction in the mixing that puts
  $e_i \in [\epsilon/2,\epsilon]$. From the rule of energy
  redistribution, we need $\epsilon/2 \leq p(e_i+e_{i+1}) \leq
  \epsilon$. Rearrange the terms to get $\epsilon/[2(e_i+e_{i+1})]
  \leq p \leq \epsilon/(e_i+e_{i+1})$. This is possible since
  $\epsilon \leq e_i+e_{i+1} \leq \sum_{i = 1}^{N}C_{i}$. Hence probabilities
  of $S_i$ are strictly positive. 
\item[(c)] There is also a uniform upper bound on $H$ given by $2\sum_{i=1} ^N C_i$. 
\item[(d)] Let $p \in (0,1)$ be the fraction in the mixing that puts $e_j
  \in (e_j, e_j+\mathrm{d}e_j)$. From the rule of energy redistribution, we
  have $e_j/(e_{j-1}+e_j) \leq p \leq (e_j
  +de_j)/(e_{j-1}+e_j)$. Because of $\epsilon \leq e_{j-1}+e_j \leq
  \sum_{i = 1}^{N}C_{i}$ and $\mathrm{d}e_j >0$, $\mathbb{P}[e_j(N\delta + j \delta) \in [e_{N+2 - j},
  e_{N+2-j}+\mathrm{d}e_j]] > \alpha \mathrm{d}e_{j}$ for some strictly positive
  constant $\alpha$. Hence probabilities
  of $F_j$ are greater than $\mathrm{const} \cdot \mathrm{d}e_{j}$. 
\end{enumerate}
In addition, all these probabilities are uniformly bounded from below
for all $\mathbf{E}$ and $\bar{\mathbf{E}}$ in $K$. Hence we have
$$
  \mathbb{P}[ S_{1} \cdots S_{N}F_{1} \cdots F_{N+1}] \geq \sigma
  \mathrm{d}e_{1}\cdots \mathrm{d}e_{N}
$$ 
for some constant $\sigma > 0$.
\end{proof}

As a corollary, we can prove that $\mathbf{E}_{n}$ is both aperiodic and irreducible 
with respect to the Lebesgue measure.

\begin{cor}
\label{aperiod}
$\mathbf{E}_{n}$ is a strongly aperiodic Markov chain.
\end{cor}
\begin{proof}
By theorem \ref{urs1}, $K$ is a uniform reference set. In addition $U_{K}(K) > 0$. The strong aperiodicity follows from
its definition.
\end{proof}

Therefore $\mathbf{E}_{n}$ is aperiodic. 

\begin{cor}
\label{irreducible}
$\mathbf{E}_{n}$ is $\lambda$-irreducible, where $\lambda$ is the
Lebesgue measure on $\mathbb{R}^{N}_{+}$. 
\end{cor}
\begin{proof}
Let $A \subset \mathbb{R}^{N}_{+}$ be a set with strictly positive
Lebesgue measure. Then there exists a set $K$ that has the form $\{(e_{1},
\cdots, e_{N} \} \,|\, 0 < c_{i} \leq e_{i} \leq C_{i} , i = 1 \sim N
\}$ and $U_K(K \cap A) > 0$. 

For any $\mathbf{E}_{0} \in \mathbb{R}^{N}_{+}$ and the time
step $h > 0$, we can choose a $K \subset \mathbb{R}^{N}_{+}$ of the form $K = \{(e_{1},
\cdots, e_{N} \} \,|\, 0 < c_{i} \leq e_{i} \leq C_{i} , i = 1 \sim N
\}$ for some $c_{i} > 0$ and $C_{i} < \infty$, such that
$\mathbf{E}_{0} \in K$. Same construction as in Theorem \ref{urs1} implies that $P^{h}(\mathbf{E}_{0},
\cdot) > \eta U_{K}( \cdot)$ for some $\eta > 0$. Therefore, $P^{h}(
\mathbf{E}_{0}, A) > \eta U_{K}(A)  > 0$. 

\end{proof}

Hence assumption {\bf (A1)} and {\bf (A2)} are satisfied.

\subsection{Absolute continuity of invariant measure.}
This subsection aims to prove the absolute continuity of $\pi$ with
respect to the Lebesgue measure, which is denoted by $\lambda$. 

\begin{pro}
\label{abscont}
If $\pi$ is an invariant measure of $\mathbf{E}_{t}$, then $\pi$ is
absolutely continuous with respect to $\lambda$ with a strictly positive density. 
\end{pro}
For
$\mathbf{E} \in \mathbb{R}^{N}_{+}$ and $t > 0$, we have decomposition 
$$
  P^{t}( \mathbf{E}, \cdot) = \nu_{\perp} + \nu_{abs} \,,
$$
where $\nu_{abs}$ and $\nu_{\perp}$ are absolutely continuous and
singular component with respect to $\lambda$,
respectively. We need to show that an absolutely continuous component
cannot revert back to singularity as time evolves. 

\begin{lem}
\label{abscont2}
For any probability measure $\mu \ll \lambda$, $\mu P^{t} \ll \lambda$
for any $t  >0$.  
\end{lem}
\begin{proof}
This proof is essentially identical to that of Lemma 6.3 of
\cite{li2014nonequilibrium}. All what we need to prove is that, for any
absolutely continuous initial distribution, the push-forward measure
corresponding to one clock ringing is still absolutely continuous. We
refer readers to \cite{li2014nonequilibrium} for the detailed
calculations. 
\end{proof}

\begin{proof}[{\bf Proof of Proposition \ref{abscont}. }] Let $\pi =
  \pi_{abs} + \pi_{perp}$ be
  an invariant measure. Assume $\pi_{perp} \neq 0$. For $t > 0$,
  $\pi_{abs}P^{t} \ll \lambda$ by Lemma \ref{abscont2}. By Theorem
  \ref{urs1}, for any $\mathbf{E} \in \mathfrak{C}$, 
$P^{t/2}( \mathbf{E}, \cdot)$ has a strictly positive density on
$\mathfrak{C}$. Since $\mathfrak{C}$ is accessible within finitely many
energy exchanges, $P^{t/2}( \mathbf{E}, \mathfrak{C} ) > 0$ for all $\mathbf{E}
\in \mathbb{R}^{N}_{+}$. Hence $P^{t}( \mathbf{E}, \cdot)$ has a strictly positive
density on $\mathfrak{C}$ for all $\mathbf{E} \in \mathbb{R}^{N}_{+}$. Therefore,
$\pi_{\perp}P^{t}$ has an absolutely continuous component. The absolutely continuous component of $\pi P^{t}$ is strictly larger
than that of $\pi$. This contradicts with the invariance of $\pi$. 

\end{proof}

\subsection{Verifying {\bf (N1)} and {\bf (N2)}.}

Now we are ready to present our numerical results. We let $N = 3$ in
our simulations. The uniform
reference set $\mathfrak{C}$ is chosen as 
$$
  \mathfrak{C} = \{ (e_{1}, \cdots, e_{N}) \,|\, 0.1 \leq e_{i} \leq
  100 , i = 1 \sim N\} \,.
$$

Throughout our numerical justification, we let $h = 0.1$. (Recall that
for a time-continuous Markov process $\Psi_{t}$, the definition of
$\tau_{\mathfrak{C}} = \tau_{\mathfrak{C}}(h)$ depends on $h$.) We will
verify {\bf (N1)} for the numerical invariant measure, which is 
generated by running the process for a sufficiently long time from a
suitable initial distribution. In our simulation, the initial
distribution for the simulation of the numerical invariant measure is
$\mu_{0} \sim (\rho_{1} , \cdots, \rho_{N} )$, where $\rho_{i}$ is an exponential
distribution with mean $(T_{L } + T_{R})/2$. We find that when $T =
200$, expectations of many observables we have tested are
stabilized (Figure \ref{kmpsstest} ). Therefore, the numerical invariant measure is chosen as
$\hat{\pi} := \mu_{0} P^{200}$. Then from the result of our simulation that chooses
the numerical invariant measure as the specific starting state, we conclude that
$\mathbb{P}_{\hat{\pi}}[\tau_{\mathfrak{C}} > t] \sim t^{-2}$. 
(Figure \ref{KMP_ss} ).

\begin{figure}[htbp]
\centerline{\includegraphics[width = \textwidth]{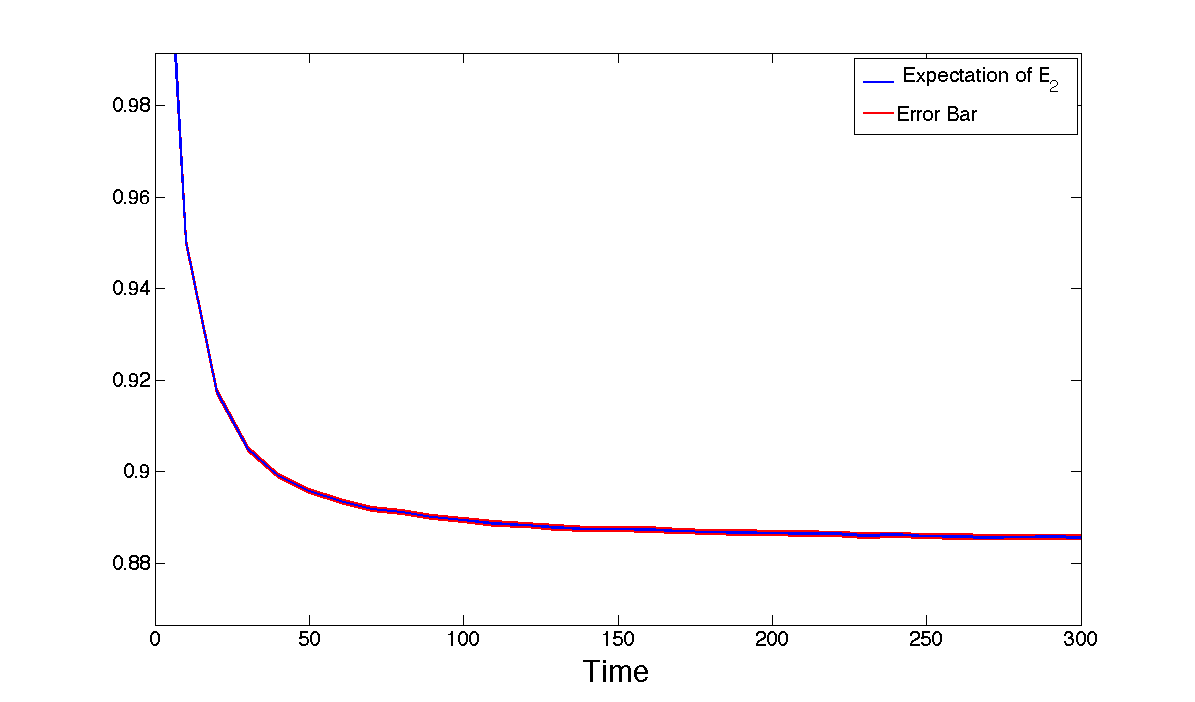}}
\caption{Expectation of energy at site $2$ verses
  time. Sample size  $M = 1 \times 10^{8}$. The error bar represents
  $\pm 1.96 M^{-1/2}$ times the sample standard deviation. }
\label{kmpsstest} 
\end{figure}

It remains to check {\bf (N2)}. The transition kernel of the time-$h$
sample chain of $\mathbf{E}_{t}$ does not have a clean explicit
form. In addition, the transition kernel of $\mathbf{E}_{t}$ has many
singularities. As a result, proving the condition in Proposition \ref{dominate} is an
extremely tedious work. Instead, we choose to numerically show that 
$$
  \gamma (\mathbf{E}) = \sup_{t \geq h}\frac{\mathbb{P}_{\mathbf{E}}[\tau_{\mathfrak{C}}
    > t]}{t^{-2}}
$$
is uniformly bounded on $\mathfrak{C}$. We follow procedure (a)-(d) in
Section 3 to show the boundedness of $\gamma (\mathbf{E})$. In fact, 
$$
  \gamma_{N}(\mathbf{E}) = \sup_{1 \leq n \leq N} \sup_{t \geq h}\frac{\mathbb{P}_{\mathbf{E}}[\tau_{\mathfrak{C}}
    > t]}{t^{-2}} 
$$
is stabilized very fast with increasing $N$. A sample of size $10^6$ is sufficient for a reliable
estimate of $\gamma (\mathbf{E})$. Figure \ref{KMP_search} shows that
when $E_{i}$ is small, $\gamma (\mathbf{E})$ decreases monotonically
with decreasing $E_{i}$ for each $i = 1 \sim 3$. Therefore, we expect that the maximal of $\gamma(\mathbf{E})$ in $\mathfrak{C}$
is reached at $\mathbf{E}_{*} = (0.1, 0.1, 0.1)$. In fact, intuitively one should expect
$\gamma(\mathbf{E})$ to decrease with site energy as starting from low site energy
means having higher probability to have even lower site energy after an 
energy exchange. Finally, we run the simulation again to estimate
$\gamma( \mathbf{E}_{*})$. As seen in Figure \ref{KMP_max}, when
starting from $\mathbf{E}_{*}$ the probability of return has tail
$\sim t^{-2}$. ($\gamma( \mathbf{E}_{*})$ is approximately $48.8915$ with standard deviation $0.0225$.)

\begin{figure}[h!]
\centering{\includegraphics[width = 0.6\paperwidth]{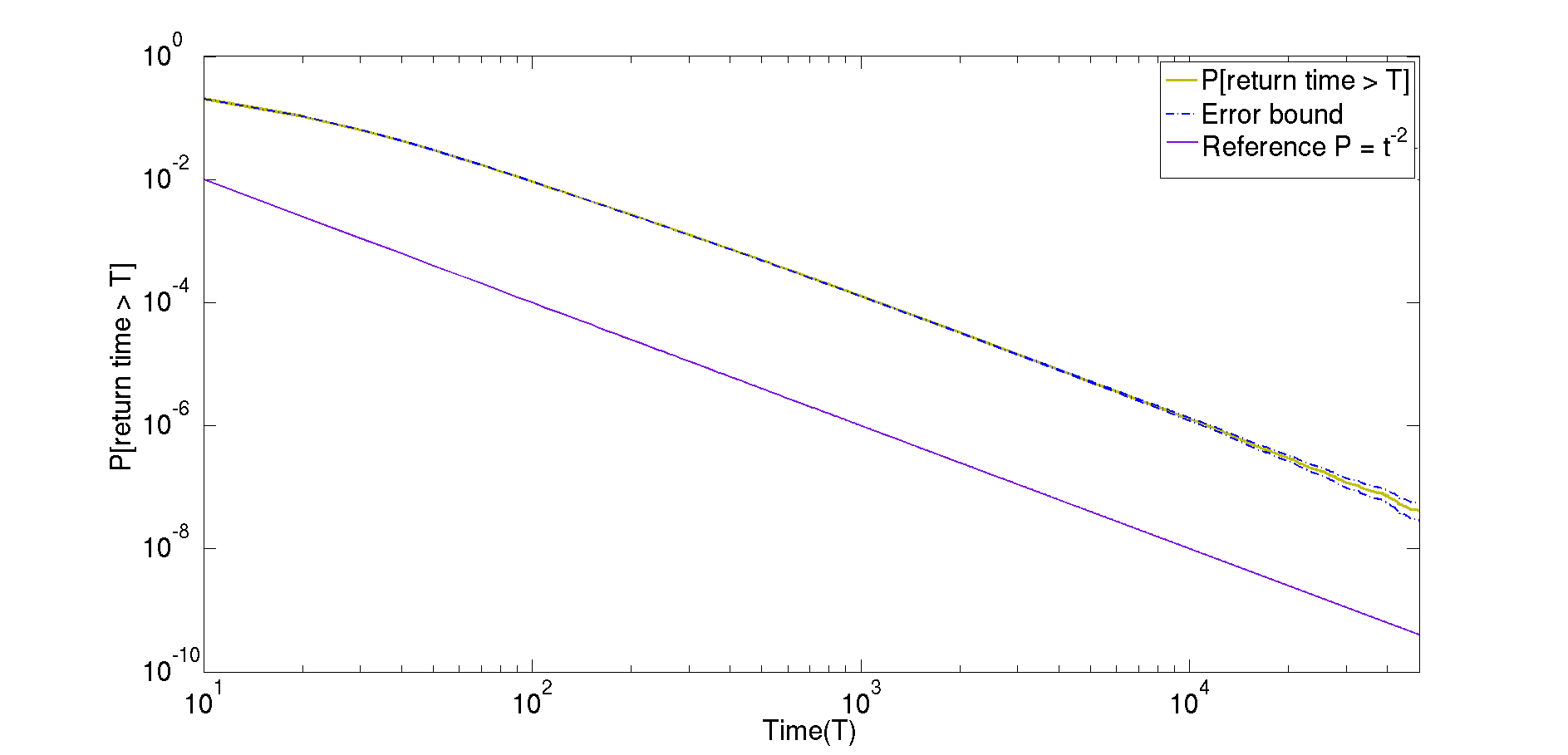}}

\caption{$\mathbb{P}_{\hat{\pi}} [\tau_{\mathfrak{C}}>t]$, $\pi$ is the
  numerical steady state obtained from long time averaging. Sample
  size $ = 1 \times 10^{8}$. The purple line is a reference line with slope
$-2$. The error bar is calculated as in Remark 3.5.}
\label{KMP_ss}
\end{figure}
\vfill
\begin{figure}[h!]
\begin{minipage}{0.5\textwidth}
\centering{\includegraphics[width = 0.38\paperwidth]{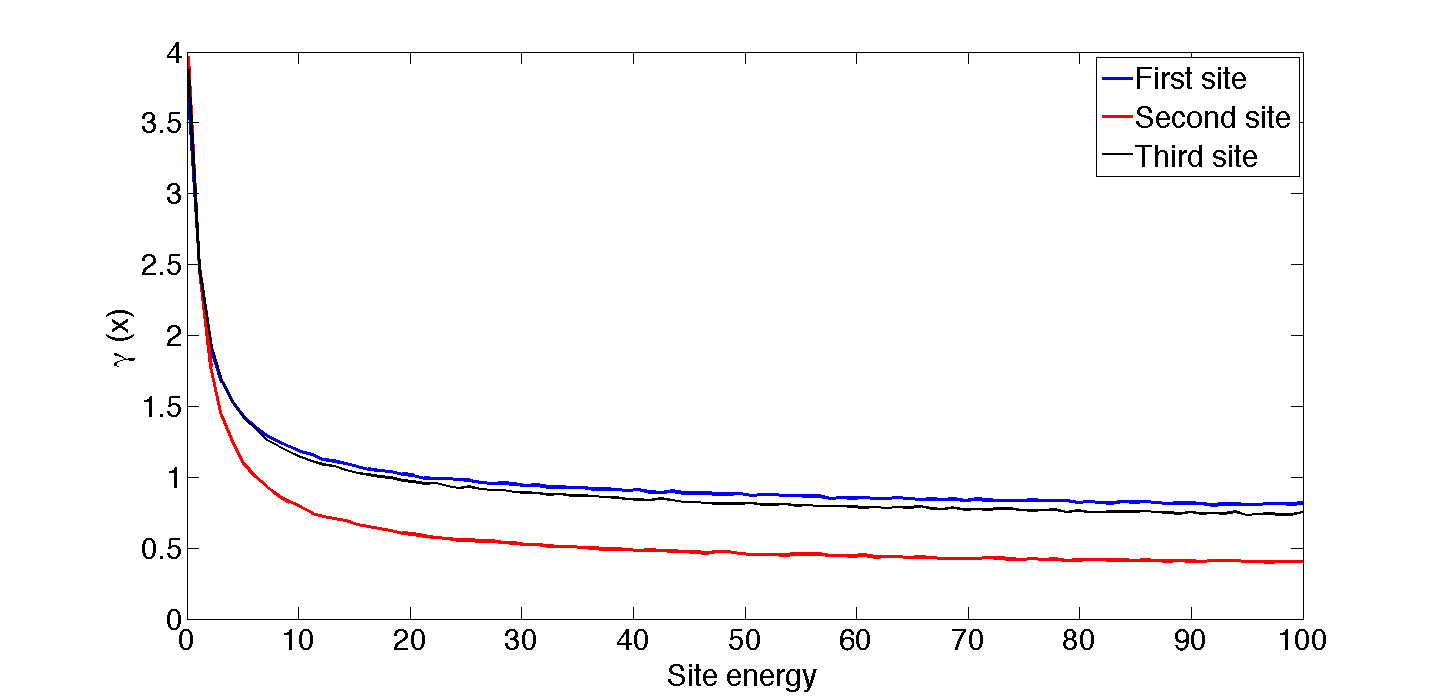}}

\caption{Change of $\gamma (\mathbf{E})$ for
  varying $\mathbf{E}$ when only one site energy changes. The
  unchanged site energy is set to be $1$. Sample size $=
  10 \times 1 \times 10^{7}$ for each initial condition.}
\label{KMP_search}
\end{minipage}\hfill
\begin{minipage}{0.5\textwidth}
\centering{\includegraphics[width = 0.38\paperwidth]{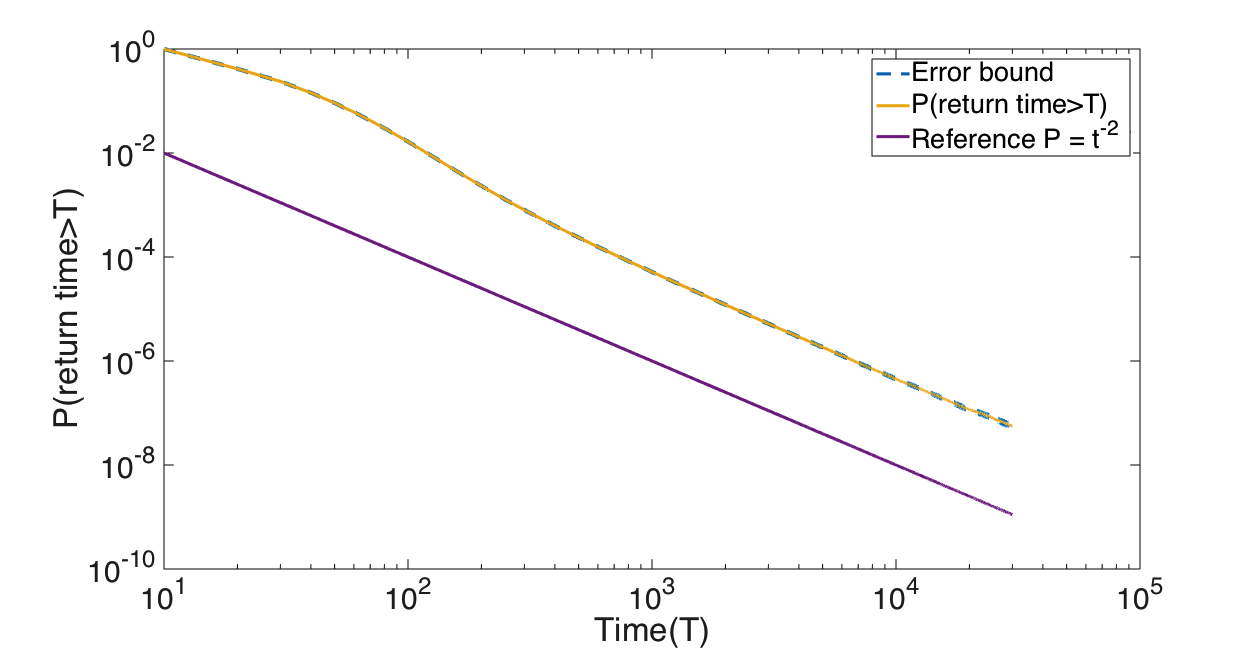}}
\caption{$\mathbb{P}_{\mathbf{E}_{*}}[\tau_{\mathfrak{C}}>t]$ for
  $\mathbf{E}_{*}=(0.1,0.1,0.1)$. Sample size $ = 1 \times 10^{10}$. The purple line is a reference line with slope
$-2$. The error bar is calculated as in Remark 3.5. }
\label{KMP_max}
\end{minipage}\hfill
\end{figure}

\subsection{Main conclusions.} 

The previous subsection verifies both {\bf(N1)} and {\bf(N2)} for $\bf{E}_t$ with parameter 2. The slope of $\mathbb{P}_{\hat{\pi}}[
 \tau_{\mathfrak{C}} > t]$ in the log-log plot is $2$. Note
 that $\mathbb{P}[\Psi_{t} = \Psi_{0} \,|\, \Psi_{0} \in K] =
 \mathbb{P}_{\Psi_{0}}[ \mbox{ no clock rings up to } t]$ is
 uniformly positive for each given $t > 0$. By Theorem \ref{c2d}, {\bf (N1)} and {\bf
  (N2)} hold for $\mathbf{E}_{n}$ with parameter $2 -
\epsilon/2$ for arbitrarily small $\epsilon > 0$. Therefore,
conclusions (a)-(d) in Section 4.4 hold for $\mathbf{E}_{n}$. 
 
It remains to pass the results for $\mathbf{E}_{n}$ to
$\mathbf{E}_{t}$. By Proposition \ref{c0}, it is sufficient to prove
``continuity at zero'' for $\mathbf{E}_{t}$. 

\begin{lem}
\label{kmpcont0}
For any probability measure $\mu$ on $\mathbb{R}^{N}_{+}$, 
$$
  \lim_{\delta \rightarrow 0} \| \mu P^{\delta} - \mu \|_{TV} = 0
$$
\end{lem}
\begin{proof}
It is sufficient to prove that for any $\epsilon > 0$, there exists a
$\delta > 0$ such that
$$
  \| \mu P^{\delta} - \mu \|_{TV} \leq \epsilon \,.
$$
Since $\mu$ is finite, there exists a bounded set $A\subset
\mathbb{R}^{N}_{+} := \{ 0 \leq E_{i} \leq M \}$ such that $\mu(A) > 1
- \epsilon/4$. By the definition of $A$, clock rates for initial values in $A$ are uniformly
bounded. Therefore, one can find a sufficiently small $\delta > 0$,
such that $\mathbb{P}[ \mbox{ no clock rings on } [0, \delta) ] \geq 1
- \epsilon/4$. For any set $U \subset \mathbb{R}^{N}_{+}$, we have
\begin{eqnarray*}
(\mu P^{\delta}) (U)  & =  & \int_{\mathbb{R}^{N}_{+}}
P^{\delta}(\mathbf{E}, U) \mu( \mathrm{d} \mathbf{E})\\
&=& \int_{A \cap U} P^{\delta}(\mathbf{E}, U) \mu( \mathrm{d}
\mathbf{E}) + \int_{A - U} P^{\delta}(\mathbf{E}, U) \mu( \mathrm{d}
\mathbf{E}) + \int_{A^{c}} P^{\delta}(\mathbf{E}, U) \mu( \mathrm{d}
\mathbf{E}) \\
&=& \mu(A \cap U) - a_{1} + a_{2} + a_{3} \,,
\end{eqnarray*}
where
\begin{eqnarray*}
a_{1} & =  & \int_{A \cap U} (1 - P^{\delta}(\mathbf{E}, U) )\mu(
\mathrm{d} \mathbf{E}) \leq \frac{\epsilon}{4} \mu( A \cap U) \leq \frac{\epsilon}{4}\\
a_{2} &=& \int_{U - A} P^{\delta}(\mathbf{E}, U) \mu( \mathrm{d}
\mathbf{E}) \leq \frac{\epsilon}{4} \mu( A - U) \leq \frac{\epsilon}{4}\\
a_{3} &=& \int_{A^{c}} \frac{\epsilon}{4} \mu( A \cap U) \leq
\frac{\epsilon}{4} \mu( A^{c}) \leq \frac{\epsilon}{4} \,. 
\end{eqnarray*}
In addition we have $\mu(U) - \mu(A \cap U) \leq \mu(A^{c}) <
\frac{\epsilon}{4}$. Hence
$$
  | (\mu P^{\delta})(U) - \mu(U) | < \epsilon
$$
for any $U \subset \mathbb{R}^{N}_{+}$. By the definition of the total
variation norm, we have
$$
  \| \mu P^{\delta} - \mu \| \leq \epsilon \,.
$$

This completes the proof. 

\end{proof}

We haven't talked about uniqueness so far. Usually the uniqueness of
the invariant probability measure
follows from the fact that $P(x, \cdot)$ admits positive density
everywhere. 

\begin{pro}
For any $h > 0$, $\mathbf{E}^{h}_{n}$ admits at most one invariant
probability measure.
\end{pro}
\begin{proof}
By the proof of Theorem \ref{urs1}, for any $\mathbf{E} \in K$,
$P^{h/2}( \mathbf{E}, \cdot)$ has strictly positive density on $K$. In
addition, $P^{h/2}(\mathbf{E}_{0}, K) > 0$ for any $\mathbf{E}_{0} \in
\mathbb{R}^{N}_{+}$. Hence $P^{h}(\mathbf{E}_{0}, \cdot)$ has positive
density on $K$. This implies that every $\mathbf{E}_{0} \in
\mathbb{R}^{N}_{+}$ belongs to the same ergodic component. Therefore
$\mathbf{E}^{h}_{n} $ cannot have more than one invariant probability measure.
\end{proof}

\bigskip 

In summary, we have the following conclusions for $\mathbf{E}_{t}$.

\begin{enumerate}
  \item For any $T_{L}$, T$_{R}$, there exists a unique invariant probability measure $\pi$,
    i.e., the nonequilibrium steady-state, which is absolutely
    continuous with respect to the Lebesgue measure on
    $\mathbb{R}^{N}_{+}$. 
\item For almost every $\mathbf{E}_{0} \in \mathbb{R}^{N}_{+}$ and any
  sufficiently small $\epsilon > 0$, we have
$$
  \lim_{t\rightarrow \infty}t^{2 - \epsilon} \|
  \delta_{\mathbf{E}_{0}} P^{t} - \hat{\pi}P^{t} \|_{TV} = 0 \,.
$$
\item For any functions $\eta$, $\xi \in L^{\infty}(
  \mathbf{R}^{N}_{+})$, we have
$$
  C_{\pi}^{\eta, \xi}(t) \leq O(1) \cdot t^{\epsilon - 2}
$$
for any $\epsilon > 0$. 
\end{enumerate}

\section{Example: random halves model}
\subsection{Derivation from deterministic dynamics} Another way to
simplify the multi-body problem as described in Figure \ref{tube} at the beginning of the previous section is
to assume that particles do not interact directly. Instead, we divide the
tube into a chain of $N$ cells, each of which contains a rotating disk
that plays the role of the ``local environment''. As seen in Figure
\ref{rhm}, particles can only exchange energy with the rotating
disk. Then we connect this chain with two thermalized ends, called
heat baths, such that thermalized particles can be injected into the system and particles in
the system can exit by entering the heat bath. This is the Hamiltonian
model proposed in \cite{eckmann2006nonequilibrium}. We refer \cite{eckmann2006nonequilibrium, li2014nonequilibrium} for details.

\begin{figure}[htbp]
\centerline{\includegraphics[width = \linewidth]{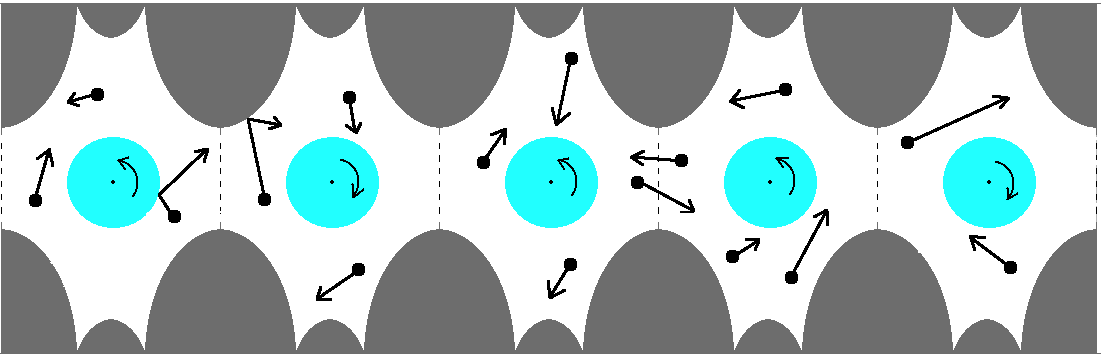}}
\caption[]{\label{rhm} Moving particles do not interact with each
  other directly. There is a rotating disk with fixed center in each cell. The collision between particles and
disks are determined by the conservation of kinetic energy and angular
momentum.}
\end{figure}

A particle in this Hamiltonian model has chaotic trajectories and
quick loss of ``memory''. Therefore, it is natural to assume that the
movement of each particle is stochastic, i.e., governed by an energy
dependent exponential clock. When a clock associated
  with a particle rings, the particle either jumps to neighboring
  cells or exchanges energy with the local environment. The
  probability of occurence of either event is a constant determined by
  the system. This reduces the Hamiltonian 
model to the so-called random halves model, which is described in the next
subsection. We refer \cite{li2014nonequilibrium} for the full detail of this model reduction process and a brief justification of the model reduction.

\subsection{Model Description}

Consider $N$ linearly ordered lattice sites$\{1,2,\cdots, N\}$, each
containing an energy tank and storing a finite number of particles
with certain amount of energy. The lattice sites are connected to two
heat baths at the ends, denoted as sites $0$ and $N+1$ for the sake of simplicity. The heat baths have temperatures $T_L$ and $T_R$, as well as exponential particle injection rates of $\rho_L$
and $\rho_R$, respectively. Particle energies are random variables with
i.i.d. distributions with a probability density function
$$\frac{2}{\sqrt{\pi}T^{3/2}}\sqrt{x}e^{-\frac{x}{T}} $$ 
where $T$ is the temperature of the heat bath from which the particle
is emitted. Notice that when a particle is emitted by the left heat
bath, it instantaneously appears at site 1. When a particle is emitted
by the right heat bath, it instantaneously appears at site N. 

An exponential clock is associated with each particle in the system,
with rate $(1+m)S\sqrt{x}$, where $x$ is the energy of the particle
and $m$ and $S$ are system constants. When the clock of a particle rings, the particle jumps with
probability $\frac{1}{1+m}$ and ``mixes'' with probability
$\frac{m}{1+m}$. Rules for jumping and ``mixing'' for a particle at site
$i$ carrying energy $x_j^i$, where $j$ is the index of the particle in
its site, are as follows.  

When a particle jumps, it goes to either site $i+1$ or site $i-1$ with
equal probability of $\frac{1}{2}$. Let $k_i$ be the number of
particles at site $i$. Then $k_i$ decreases by 1 while the site that
the particle jumps to has an increase in particle number by 1. Notice
that the particle leaves the system if it jumps to the left or right
heat bath. For ``mixing'', we mean a particle exchanges energy with the stored
energy at its site. Let $s_i$ be the stored energy at site $i$ and let
$x_j$ be the corresponding particle energy.  The rule of energy
exchange is $(s_i ', x'_j)=(x_j u^2, s_i +x_j (1-u^2))$, where u is a
uniform random variable distributed on $(0,1)$.   

Random halves model generates a Markov jump
process $$\mathbf{\omega}_t = ((s_1(t),\{x_1^1(t),\ldots,
x_{k_1}^1(t)\}), \ldots, (s_N(t),\{x_1^N(t),\ldots,
x_{k_N}^N(t)\})),$$ where $k_1, \ldots, k_N$ take values in
$\mathbb{N} \cup \{0,\infty\}=\{0,1,\ldots, \infty\}$. The stochastic
process takes values in the state space $\Omega = \Pi_{i=1}^N
\Omega^i$, where $\Omega^i = \cup_{k=0}^\infty \Omega_k^i$, and
$\Omega_k^i = \{(s_i, \{x^i_1, \ldots, x^i_k\})|s_i, x^i_j>0\}$. Since
we regard particles as indistinguishable to avoid confusion when
particles re-enter the system, we use unordered lists denoted by curly
brackets. Notice that $\Omega_k^i=\mathbb{R}_+ \times
((\mathbb{R}_+)^k /\sim)$, where $\sim$ is the equivalence relation
given by $\{x_1, \ldots, x_k\} \sim \{x_{\sigma_{1}}, \cdots,
x_{\sigma_{k}}\}$, $\sigma$ is any $k$-permutation.

Therefore, we can define the Markov jump process $\omega(t)$ generated by random halves model on $\Omega$. We denote the transition kernel of $\omega(t)$ by $P^t(\mathbf{\omega}, \cdot)$. Let $h > 0$ be a fixed number that represents the step size. The
time-$h$ sampling chain of $\mathbf{\omega}_{t}$ is denoted by
$\mathbf{\omega}^{h}_{n}$, or simply $\mathbf{\omega}_{n}$ when it does not lead to a
confusion. The transition kernel of $\mathbf{\omega}_{n}$ is denoted by
$P(\mathbf{\omega}, \cdot)$. 

For the sake of later use, we will define a reference measure
$\Lambda$ on $\Omega$, where $\Lambda = \prod_1^N \Lambda_i$ and
$\Lambda_{i}$ is the natural reference measure on $\Omega^{i}$, such
that the restriction of $\Lambda_{i}$ on $\Omega^{i}_{k}$ is the
quotient of the Lebesgue measure on $\mathbb{R} \times
\mathbb{R}^k$ under the relationship $\sim$.
  
\subsection{Verifying \bf{(A1)} and \bf{(A2)}.}
We will first work on the time-$h$ sampling chain of $\mathbf{\omega}_n$. The verification of analytical conditions of $\mathbf{\omega}_n$ is based on the following theorem. 
\begin{thm}
\label{analytic}
For any set $K \subset \Omega$, of the form $\{\omega \in \Omega|0 \leq k_i \leq K_0, 0\leq s_i \leq S_0, c \leq x^i_j \leq C, i=1 \sim N, j=1 \sim k_i \}$, where $K_0, S_0, c, C$ are positive constants, and any $h>0$, there exists a constant $\eta>0$ such that 
$$
P(\omega,\cdot)>\eta \Lambda_K(\cdot),
$$
for all $\omega \in K$, where $\Lambda_K$ is the reference measure
restricted to $K$.
\end{thm}
\begin{proof}
For $\omega = ((\hat{s}_1,\{\hat{x}_1^1,\ldots,
\hat{x}_{k_1}^1\}), \ldots, (\hat{s}_N,\{\hat{x}_1^N,\ldots,
\hat{x}_{k_N}^N\})) \in K$, let $A_{\omega}(ds)=
\{(s_1,\{x_1^1,\ldots, x_{k_1}^1\}), \ldots,
(s_N,\{x_1^N,\ldots, x_{k_N}^N\})) | \hat{s}_1 \leq s_1 \leq
\hat{s}_1+ds, \ldots, \hat{s}_N \leq s_N \leq \hat{s}_N +ds,
\hat{x}_j^i \leq x_j^i \leq \hat{x}_j^i +ds, \text{ for } 1\leq i \leq
N, 1\leq j \leq k_i   \}$ for $ds \ll 1$. 

It then suffices to prove
that for each
$$
  \bar{\omega} = ((\bar{s}_1,\{\bar{x}_1^1,\ldots,
\bar{x}_{\bar{k}_1}^1\}), \ldots, (\bar{s}_N,\{\bar{x}_1^N,\ldots,
\bar{x}_{\bar{k}_N}^N\}))  \in K \,,
$$
we have
$$
P(\bar{\omega}, A_\omega (ds)) \geq \sigma ds^{1 + k_{1}}\cdots ds^{1 + k_{N}} 
$$
for $0 < ds \ll 1$, where $\sigma$ is a strictly positive constant that is independent of
$\bar{\omega}$ and $\omega$. 

In order to do so, we will construct a sequence of events to go from
the state $\bar{\omega}$ to $A_\omega (ds)$ with positive probability. Let
$M = \sum_{i = 1}^{N} k_{i}$ and $\bar{M} = \sum_{i = 1}^{N }
\bar{k}_{i}$. Let $\delta = \frac{h}{M + \bar{M} + N}$. And let $E=
\sum_{ i = 1}^{N} (\hat{s}_i+ds)$. We consider the events $F_{i}, G, H_i, A_{i}$, where $i = 1 \sim
N$. 

\begin{itemize}
  \item $F_{i}= \{ \text{on } ( (i-1) \delta, i\delta], i = 1 \sim \bar{M}$, no new particle enters, and
    one particle present initially exits the system without exchanging
    energy $\}$. 
\item Define an auxiliary event $A_i = \{\hat{s}_i \leq s_i \leq
  \hat{s}_i +ds\}$ on $((M + i-1)\delta, (M + i) \delta]$. 
\item $H_i = A_i \cap$ $\{$ on $( M\delta+(i-1)\delta, M\delta+i
  \delta]$, exactly one particle, carrying energy on the interval $[E,
  2E]$, enters from the
  left and jumps through all sites $s_{j}, j < i$ without exchanging
  energy, then exchanges energy with site $i$, and jumps to exit the
  system from the right $\}$ $\cap$ $\{$ no other new particle enters
  $\}$ for $i=M \sim M+N$. 
\item $G = $ $\{$ on $(  i \delta, (i+1)\delta ],  i = \bar{M} + N \sim \bar{M}+N+M$,
    $k_{i}$ particles with energy $x_j^i \in [\hat{x}_j^i,\hat{x}_j^i
    +ds]$ for $j=1 \sim k_i$
    enters the system from the left and jumps until reaching site $i$,
  without exchanging energy $\}$ $\cap$ $\{$ no new particle enters and
  existing particles do nothing $\}$.
\end{itemize}

The idea is that particles initially present at each site are first
emptied from the system. Then for each site, one particle with
sufficiently large amount of energy enters the system to mix at the
corresponding site. Lastly, particles in the target set enter the
system and jump to corresponding sites. We need to show that the
probability of occurence of the sequence of events described above is
always strictly positive. Here are the considerations. 

\begin{enumerate}
\item The initial number of particles $k_i \leq K_0$ for each
  $i=1 \sim N$. Clock rates $(1+m)S \sqrt{x_j^i}$ are bounded above
  zero since $x_j^i \geq c$, and bounded below infinity since $x_j^i
  \leq C$ by assumption. Hence $F_{i}$ occurs with strictly
  positive probability. 

\item  Let $u \in (0,1)$ be the fraction in the mixing that puts $s_i
  \in [\hat{s}_i, \hat{s}_i+ds]$. Let $X$ be the particle energy in
  the event $H_{i}$. Rearrange the terms to get
  $\hat{s}_i/X \leq u^2 \leq
  (\hat{s}_i+ds)/X$. Note that $X$ is bounded from above by $2N S_{0}$. Therefore, $\mathbb{P}[ A_{i}] > \alpha
  ds$ for some constant $\alpha > 0$.

\item In addition, during the event $\{ H_{i}, i = 1 \sim N \}$, $E$
  is bounded from above by $N S_{0}$. The probability of $H_{i}$ is
  greater than $\mathrm{const} \cdot ds$.

\item The number of particles in the destination set $\bar{k}_i \leq
  K_0$ for each $i= 1 \sim N$, and clock rates are bounded both above
  from zero and below from $C$. Hence $G$ occurs with probability at
  least $\mathrm{const} \cdot ds^{k_{1} } \cdots ds ^{k_{N}}$. 
\end{enumerate}
In addition, all these probabilities and probability densities are uniformly bounded from below
for all $\bar{\omega}$ and $\omega$ in $K$. 
\end{proof}
As a corollary, we can prove that $\omega_n$ is both aperiodic and irreducible with respect to the reference measure. 
\begin{cor}
$\omega_n$ is a strongly aperiodic Markov chain. 
\end{cor}
\begin{proof}
By \ref{analytic}, $K$ is a uniform reference set. In addition $\Lambda_K(K)>0$. The strong aperiodicity follows from its definition.
\end{proof} 
Therefore $\omega_n$ is aperiodic. 
\begin{cor}
$\omega_n$ is $\Lambda$-irreducible, where $\Lambda$ is the reference measure on $\Omega$.
\end{cor}
\begin{proof}
Let $A \subset \Omega$ be a set with strictly positive measure.  Then
there exists a set $U$ that has the form $U = \{\omega \in \Omega|0 \leq
k_i \leq K_0, 0\leq s_i \leq S_0, c \leq x^i_j \leq C, i=1 \sim N, j=1
\sim k_i \}$ and $\Lambda_K(U \cap A)>0$. 

For any $\omega_0 \in \Omega$ and the time step $h>0$, there exists a
set $K$ that has the form as in Theorem \ref{analytic} such that
$\omega_{0} \in K$. Hence same
construction as in Theorem \ref{analytic} implies that $P^h(\omega_0, \cdot) > \eta \Lambda_K(\cdot)$ for some $\eta>0$. Therefore $P^h(\omega_0, A)>\eta \Lambda_K(A)>0$. 
\end{proof} 

\subsection{Absolute continuity of invariant measure}
The proof of the absolute continuity of $\pi$ with respect to
$\Lambda$ is similar as in the previous section.

\begin{pro}
If $\pi$ is an invariant measure of $\omega_t$, then $\pi$ is absolutely continuous with respect to $\lambda$ with a strictly positive density. 
\end{pro}
\begin{proof}
Let $\pi=\pi_{abs}+\pi_{\perp}$ be an invariant measure, where $\pi_{abs}$ and $\pi_{\perp}$ are absolutely continuous and singular components with respect to $\Lambda$ respectively. Since $\pi_{abs} \ll  \Lambda$, $\pi_{abs} P^t  \ll \Lambda$ for any $t > 0 $ by Lemma 6.3 of
\cite{li2014nonequilibrium}. For similar reasons as in the previous
model, we again refer readers to \cite{li2014nonequilibrium} for
detailed calculations. The rest of the proof then follows the same
line as in the proof of Proposition \ref{abscont}.
\end{proof} 

\subsection{Verifying {\bf (N1)} and {\bf (N2)}.}
Now we will present our numerical results for the random halves model. We let $N=3$ or $4$, depending on the computational cost of the simulation. The uniform reference set $\mathfrak{C}$ is chosen as 
$$
\mathfrak{C}=\{\omega \in \Omega|0 \leq k_i \leq 40, 0 \leq s_i \leq 100, 0.1 \leq x^i_j \leq 100, i =1 \sim N, j=1 \sim k_i\}
$$
Throughout our numerical justification about $\tau_{\mathfrak{C}} = \tau_{\mathfrak{C}}(h)$, we let $h = 0.1$. We very {\bf(N1)} for the numerically generated invariant measure. The
numerical invariant measure is generated by running the process for a
sufficiently long time from a suitable initial distribution. In our
simulation, the initial distribution for the simulation of the
numerical invariant measure is 
$$
\omega_0 \sim ((s_1(t),\{x_1^1(t),\ldots,
x_{k_1}^1(t)\}), \ldots, (s_N(t),\{x_1^N(t),\ldots, x_{k_N}^N(t)\})) \,,
$$
where $N=3$, each $s_i$ is uniformly distributed between 0 and 100,
each $k_i$ is a poisson distribution with mean $(\rho_L+\rho_R)/2$,
and each $x^i_j$ is an exponential distribution with mean
$(T_L+T_R)/2$. We find that when $T=100$, the expectation of the
observables we have tested are stabilized (Figure \ref{rhmsstest}). Therefore the numerical
invariant measure is chosen as $\hat{\pi} := \mu_0 P ^{100}$. Then from the result of our simulation that chooses the numerical invariant meausure as the specific starting state,
we conclude that $\mathbb{P}_{\omega_0}[\tau_{\mathfrak{C}}>t] \sim
t^{-2}$ and $\mathbb{P}_{\hat{\pi}}[\tau_{\mathfrak{C}}>t] \sim
t^{-2}$. (Figure \ref{RHM_ss}). 

\begin{figure}[htbp]
\centerline{\includegraphics[width = \textwidth]{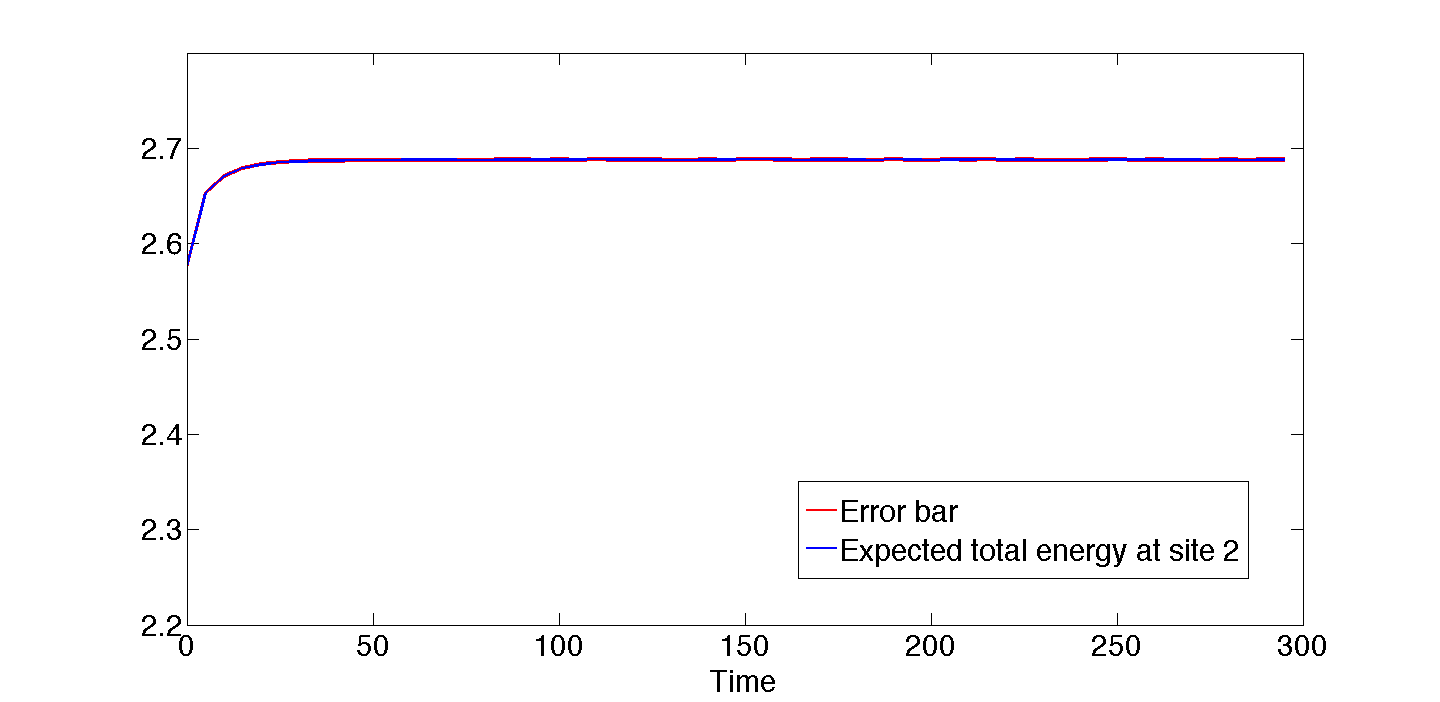}}
\caption{Total energy at the second site verses time. Sample size $ M =
  1 \times 10^{8}$. The error bar represents $\pm 1.96 M^{-1/2}$ times
the sample standard deviation. }
\label{rhmsstest}
\end{figure}

It remains to check {\bf{(N2)}}. Same as in the previous section, the
transition kernel of the time-$h$ sample chain does not have a clean
explicit form. Hence we choose to numerically verify the boundedness
of 
$$
  \gamma (\omega) = \sup_{t \geq h}
  \frac{\mathbb{P}_{\omega}[\tau_{\mathfrak{C}} > t ]}{ t^{-2}} \,.
$$
In this example, the uniform reference set $\mathfrak{C}$ has much
higher dimension. But we can still numerically capture the
monotonicity. Our simulation result shows that $\gamma (\omega)$ increases
monotonically with decreasing site energy  (Figure \ref{RHM_search_1})
and particle energy at each site (Figure \ref{RHM_search_2}), and increases
monotonically with number of particles at each site (Figure \ref{RHM_search_3}). Therefore we expect the
maximal of $\gamma (\omega)$ in $\mathfrak{C}$ to be $\omega_{*}=\{(0,\{0.1,\ldots,
0.1\}), (0,\{0.1,\ldots, 0.1\}),(0,\{0.1,\ldots, 0.1\})\}$, for
$k_i=40, i = 1\sim 3$. This result matches the heuristic argument that
lower site energy, lower particle energy, and higher number of
particles at the initial condition produces higher probability of
entering the low energy states. Finally, we run the simulation again
with initial value $\omega_{*}$ to verify that the return time has
tail $\sim t^{-2}$ in Figure \ref{RHM_max}. ($\gamma( \omega_{*})$ is
approximately $1291.8695$ with standard deviation $5.3388$.)

\begin{figure}[h!]
\centering{\includegraphics[width = 0.6\paperwidth]{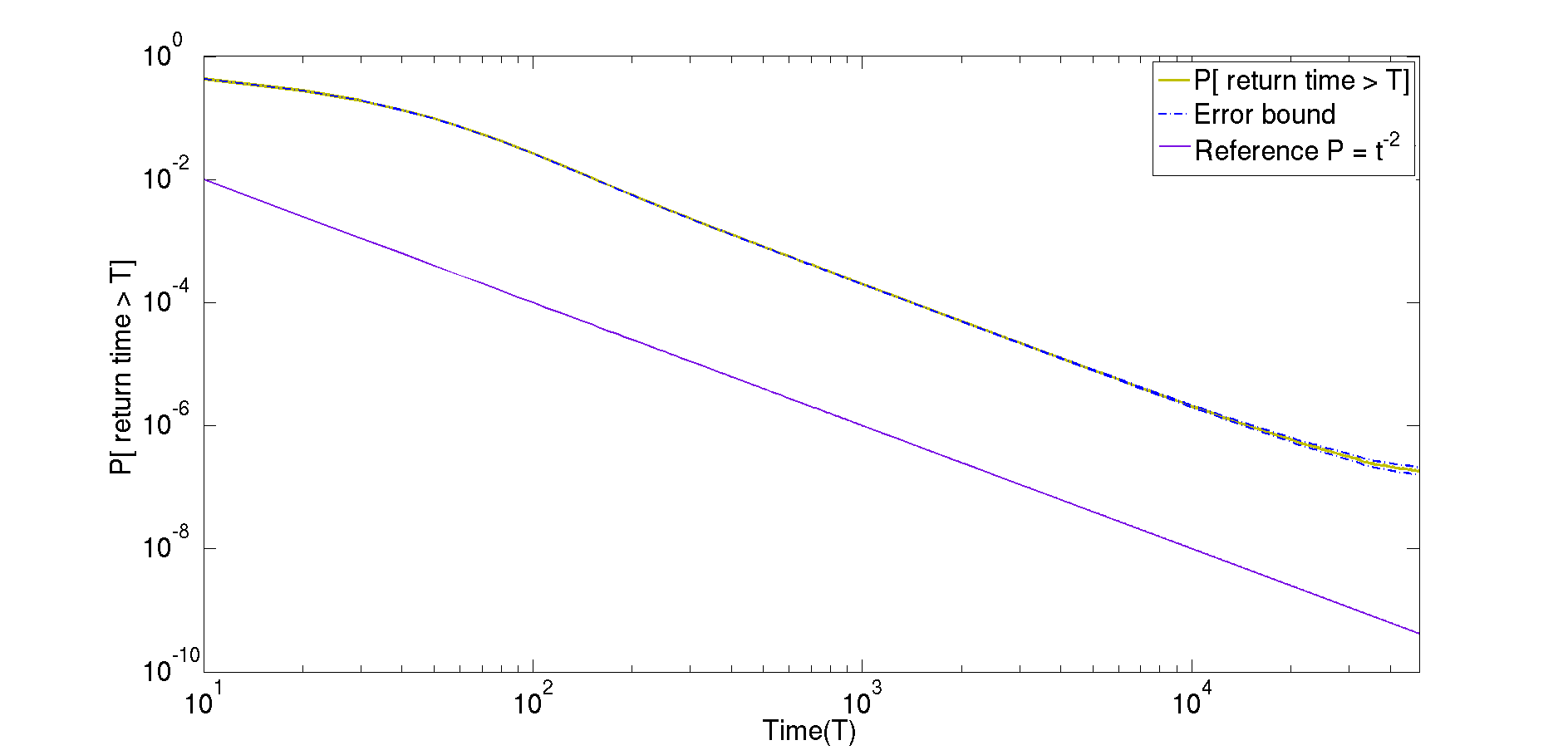}}
\caption{$\mathbb{P}_{\hat{\pi}} [\tau_{\mathfrak{C}}>t]$, $\pi$ is the
  numerical steady state obtained from long time averaging. Sample
  size = $1 \times 10^{8}$. The purple line is a reference line with slope
$-2$. The error bar is calculated as in Remark 3.5.}
\label{RHM_ss}
\end{figure}

\begin{figure}[h!]
\begin{minipage}{0.5\textwidth}
\centering{\includegraphics[width = 0.38\paperwidth]{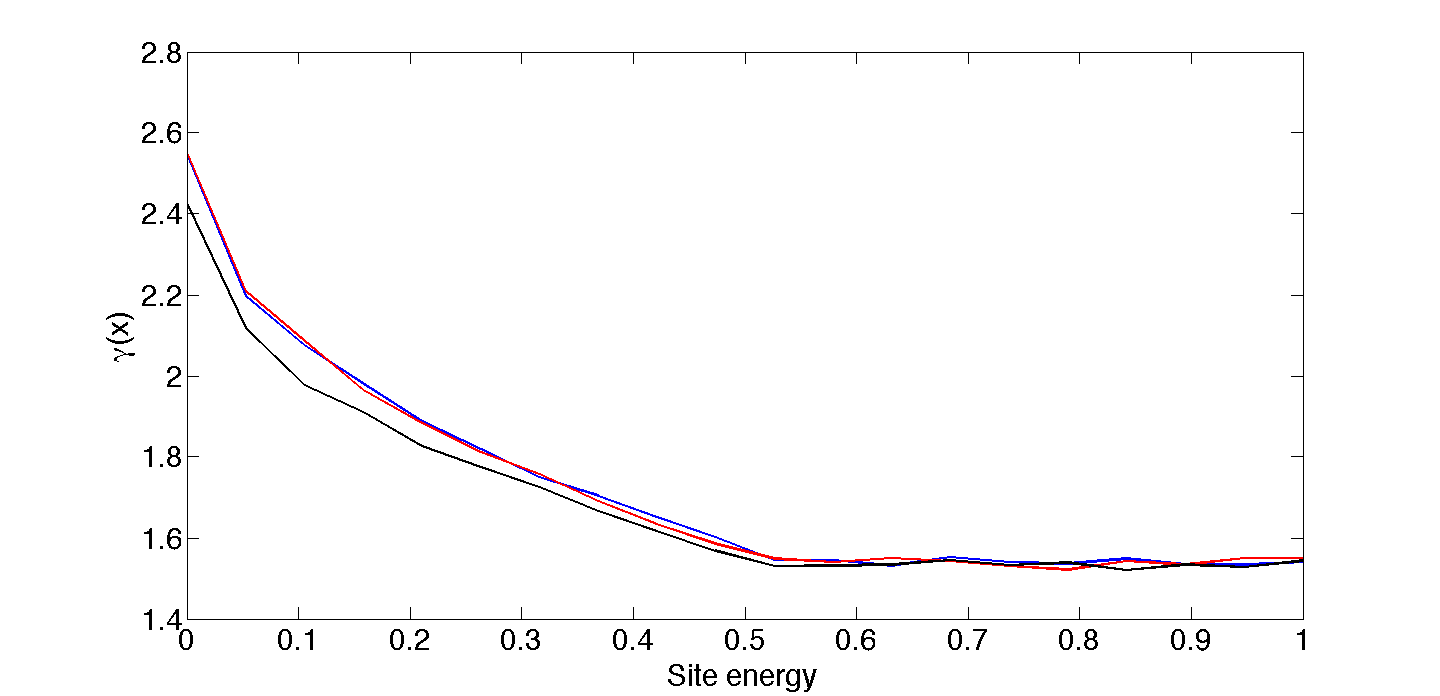}}
\caption{Change of $\gamma (\omega)$
  for varying $\omega$ when site energy at only one site changes. The
  energy configuration of an unchanged site is: site energy $ = 1$,
  two particles per site with particle energy = $1$. Sample size = $1 \times 10^{7}$ for each initial condition.}
\label{RHM_search_1}
\end{minipage}\hfill
\begin{minipage}{.5\textwidth}
\centering{\includegraphics[width = 0.38\paperwidth]{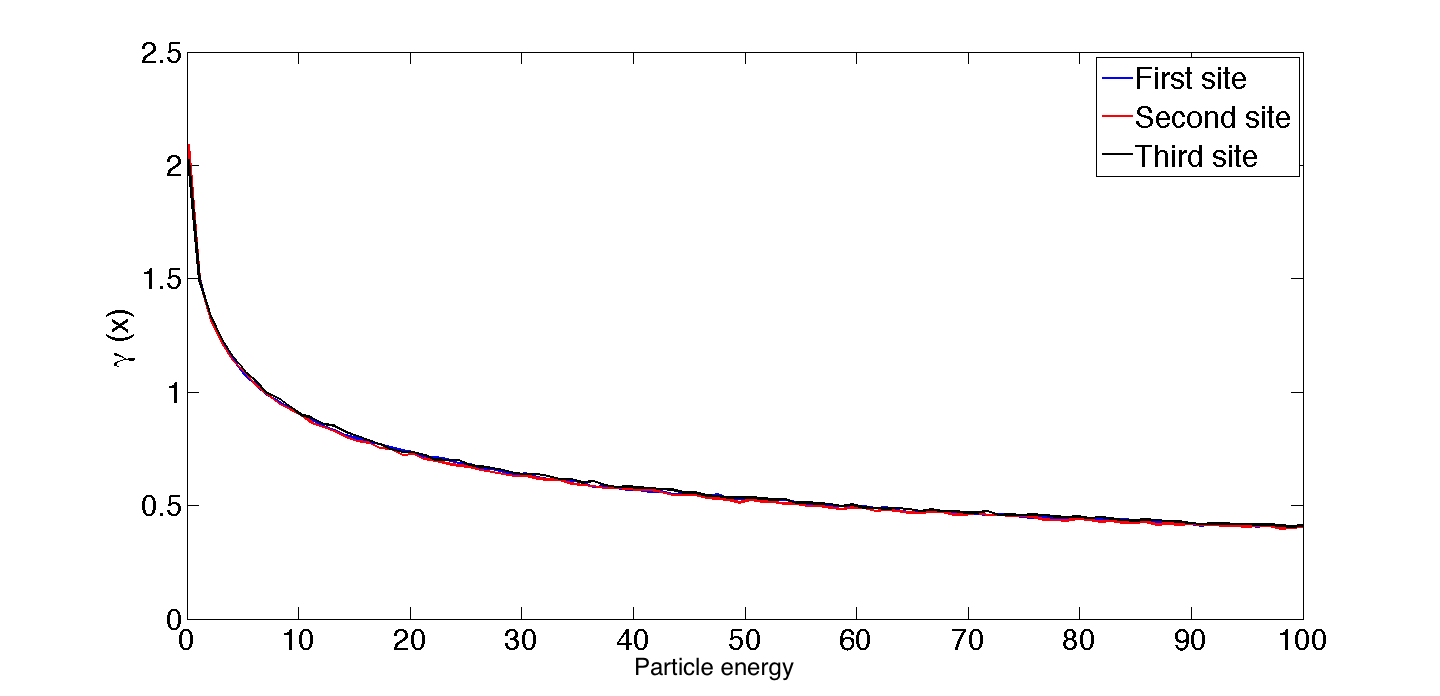}}
\caption{Change of $\gamma (\omega)$
  for varying $\omega$ when particle energy at only one site changes. The
  energy configuration of an unchanged site is: site energy $ = 1$,
  two particles per site with particle energy = $1$. Sample size = $1 \times 10^{7}$ for each initial condition.}
\label{RHM_search_2}
\end{minipage}\hfill
\end{figure}

\begin{figure}[h!]
\begin{minipage}{0.5\textwidth}
\centering{\includegraphics[width = 0.38\paperwidth]{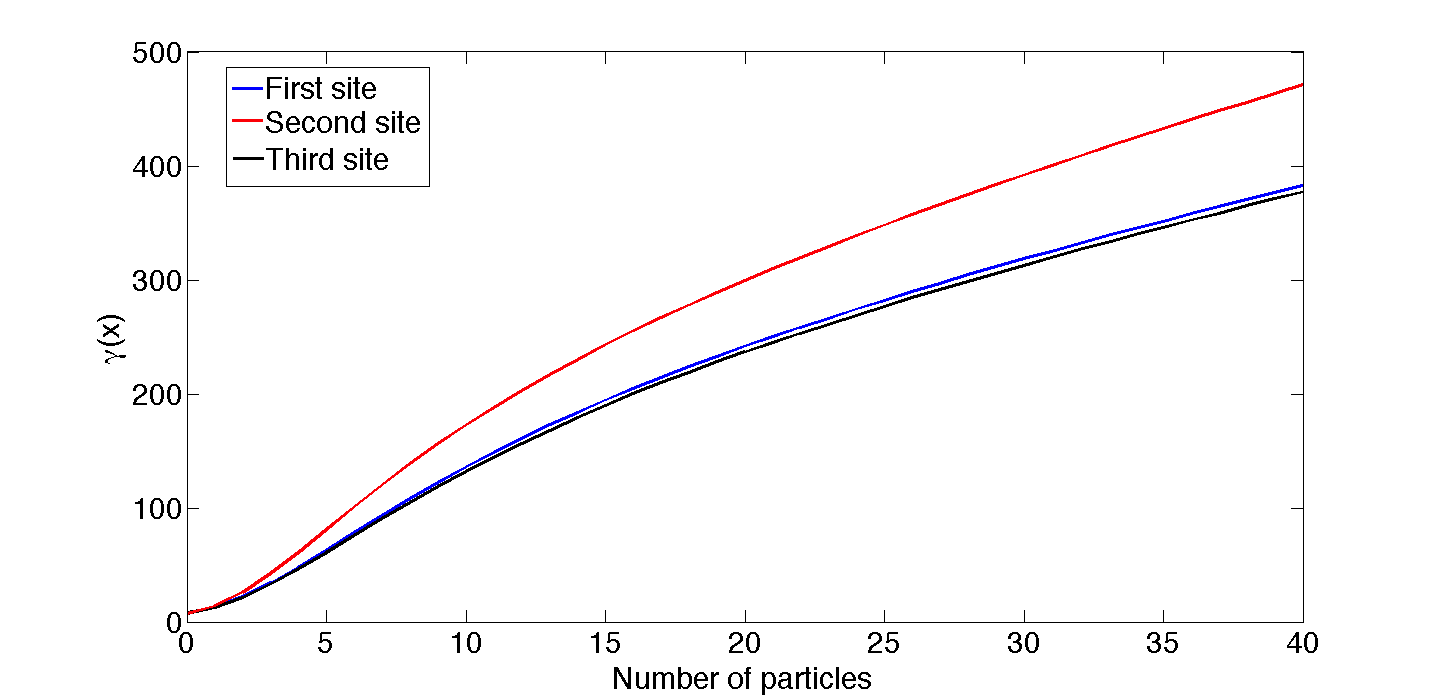}}
\caption{Change of $\gamma (\omega)$
  for varying $\omega$ when number of particles at only one site
  changes. Energy of each particle is $0.1$. The
  energy configuration of an unchanged site is: site energy $ = 1$,
  two particles per site with particle energy = $1$. Sample size = $1 \times 10^{7}$ for each initial condition.}
\label{RHM_search_3}
\end{minipage}\hfill
\begin{minipage}{.5\textwidth}
\centering{\includegraphics[width = 0.38\paperwidth]{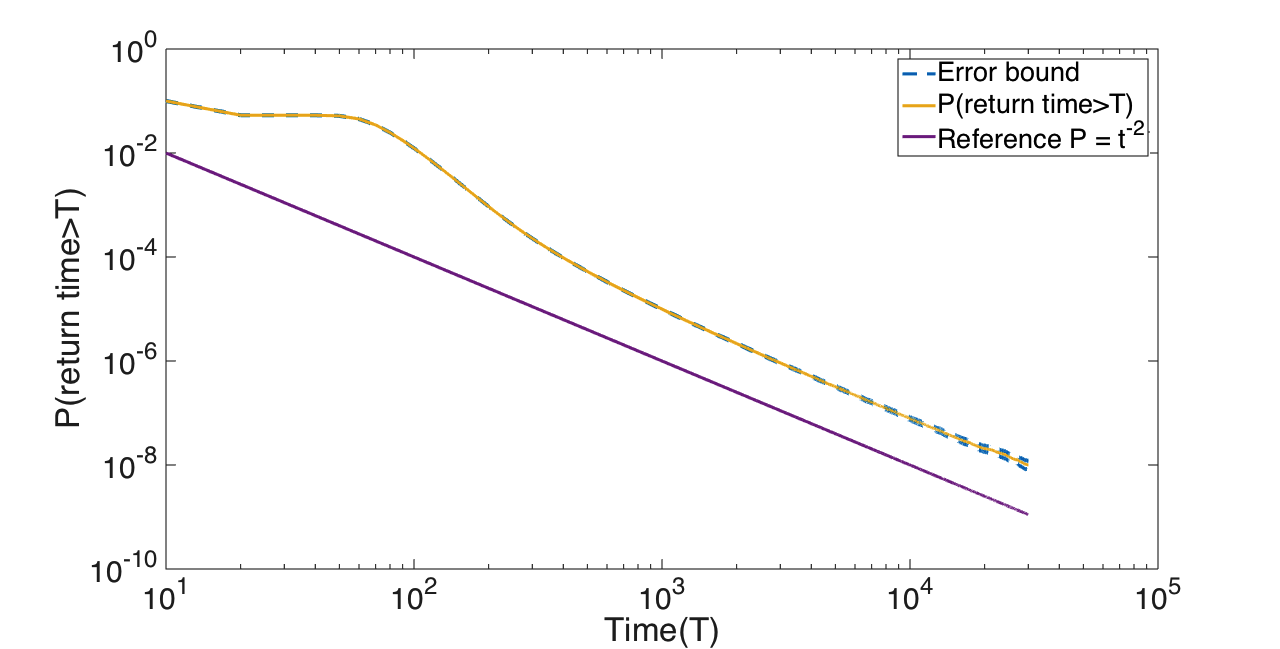}}
\caption{$\mathbb{P}_{\omega*}[\tau_{\mathfrak{C}}>t]$ for
  $\omega_{*}=\{(0,\{0.1,\ldots, 0.1\}),\ldots, (0,\{0.1,\ldots,
  0.1\})\}$, for $k_i=40, i=1 \sim 3$. Sample size = $1 \times
  10^{10}$. The purple line is a reference line with slope
$-2$. The error bar is calculated as in Remark 3.5.}
\label{RHM_max}
\end{minipage}\hfill
\end{figure}

\subsection{Main Conclusion}
The previous subsection verifies {\bf (N1)} for $\delta_{\omega_0}$
and $\hat{\pi}$, as well as {\bf (N2)}. The slopes of
$\mathbb{P}[\tau_\mathfrak{C}>t]$ for both initial conditions are
$2$. Note that $\mathbb{P}[\Psi_{t} = \Psi_{0} \,|\, \Psi_{0} \in K] =
\mathbb{P}_{\Psi_{0}}[ \mbox{ no clock rings up to } t]$ is uniformly
positive for each given $t > 0$. By Theorem \ref{c2d}, {\bf (N1)}
and {\bf (N2)} hold for $\omega_{n}$ with parameter $2
-\epsilon/2 $ for arbitrarily small $\epsilon > 0$. Therefore, conclusions (a)-(d) in Section 4.4 hold for $\omega_{n}$.   
 
It remains to pass the results for $\omega_{n}$ to $\omega_{t}$. By
Proposition \ref{c0}, it is sufficient to prove ``continuity at zero''
for $\omega_{t}$.  

\begin{lem}
For any probability measure $\mu$ on $\Omega$, 
$$
  \lim_{\delta \rightarrow 0} \| \mu P^{\delta} - \mu \|_{TV} = 0
$$
\end{lem}
\begin{proof}
It is sufficient to prove that for any $\varepsilon > 0$, there exists a
$\delta > 0$ such that
$$
  \| \mu P^{\delta} - \mu \|_{TV} \leq \epsilon \,.
$$
Since $\mu$ is finite, there exists a bounded set $A\subset \Omega := \{ 0\leq k_i \leq K, 0 \leq s_i \leq S, 0 \leq x_j^i \leq M \}$ such that $\mu(A) > 1
- \epsilon/4$. By the definition of $A$, clock rates for initial values in $A$ are uniformly
bounded. Therefore, one can find a sufficiently small $\delta > 0$,
such that $\mathbb{P}[ \mbox{ no clock rings on } [0, \delta) ] \geq 1
- \epsilon/4$. For any set $U \subset \Omega$, the same calculation as
in the proof of Lemma \ref{kmpcont0} implies
$$
  | (\mu P^{\delta})(U) - \mu(U) | < \epsilon
$$
for any $U \subset \Omega$. By the definition of the total
variation norm, we have
$$
  \| \mu P^{\delta} - \mu \| \leq \epsilon \,.
$$

This completes the proof. 

\end{proof}

It remains to prove the uniqueness of the invariant measure.

\begin{pro}
For any $h > 0$, $\omega^{h}_{n}$ admits at most one invariant
probability measure.
\end{pro}
\begin{proof}
This proof is the same as that of Proposition 5.7.
\end{proof}

\bigskip 

In summary, we have the following conclusions for $\omega_{t}$.

\begin{enumerate}
  \item For any $T_{L}$, T$_{R}$, there exists a unique invariant probability measure $\pi$,
    i.e., the nonequilibrium steady-state, which is absolutely
    continuous with respect to the reference measure on
    $\Omega$. 
\item For almost every $\omega_{0} \in \Omega$ and any
  sufficiently small $\varepsilon > 0$, we have
$$
  \lim_{t\rightarrow \infty}t^{2 - \varepsilon} \|
  \delta_{\omega_0} P^{t} - \hat{\pi}P^{t} \|_{TV} = 0 \,.
$$
\item For any functions $\eta$, $\xi \in L^{\infty}(
  \Omega)$, we have
$$
  C_{\pi}^{\eta, \xi}(t) \leq O(1) \cdot t^{\varepsilon - 2}
$$
for any $\varepsilon > 0$. 
\end{enumerate}

\end{document}